\documentclass[dvipdfmx]{article}
\usepackage{amsthm, amscd, amsmath, amsfonts, mathrsfs, float}
\newtheorem{theo}{Theorem}[section]
\newtheorem{proposition}[theo]{Proposition}
\newtheorem{lemma}[theo]{Lemma}
\newtheorem{corollary}[theo]{Corollary}
\newtheorem{rem}[theo]{Remark}

\begin{document}

\title{On Exact Triangles Consisting of Stable Vector Bundles on Tori}

\author{Kazushi Kobayashi\footnote{Department of Mathematics and Informatics, Graduate School of Science, Chiba University, Yayoicho 1-33, Inage, Chiba, 263-8522 Japan. E-mail : afka9031@chiba-u.jp}}

\date{}

\maketitle

\begin{abstract}
In this paper, we consider the exact triangles consisting of stable vector bundles on one-dimensional complex tori, and give a geometric interpretation of them in terms of the corresponding Fukaya category via the homological mirror symmetry. 
\end{abstract}

\tableofcontents

\section{Introduction}
Mirror symmetry is a symmetry between symplectic geometry and complex geometry, and homological mirror symmetry conjectured by M. Kontsevich \cite{K} is one of the most important conjectures as a mathematical formulation of mirror symmetry. The homological mirror symmetry is an equivalence between triangulated categories which are obtained by the Fukaya categories (a Fukaya category is an $A_{\infty }$ category) \cite{Fuk} and the derived categories of coherent sheaves (a derived category is a triangulated category). In this mathematical formulation, a certain triangulated category is constructed from an $A_{\infty }$ category in order to compare the Fukaya category with the derived category of coherent sheaves. Here, exact triangles in a triangulated category obtained by this construction are defined as exact triangles associated to mapping cones. Thus, the notion of a mapping cone is one of the fundamental tools in the formulation of the homological mirror symmetry. 

One of the most fundamental examples of mirror pairs is a pair $(T^2,\check{T}^2)$ of tori. In particular, for an affine Lagrangian submanifold $\mathcal{L}_{\frac{a}{n}}$, i.e., a line of rational slope $\frac{a}{n} \in \mathbb{Q}$ on a covering space of a symplectic torus $T^2$, we can define a holomorphic vector bundle $E(\mathcal{L}_{\frac{a}{n}})$ whose rank and degree are $n$ and $a$, respectively on the mirror dual complex torus $\check{T}^2$ of $T^2$. Note that those holomorphic vector bundles are stable. There are many studies of the homological mirror symmetry for tori (\cite{P}, \cite{Fuk t}, \cite{P A}, \cite{abouzaid} etc.). In this paper, we discuss the structures of exact triangles consisting of those stable vector bundles in the triangulated category.

In the homological mirror symmetry setting, we usually consider the derived category of coherent sheaves in the complex geometry side, but in the present paper we instead consider a DG-category consisting of holomorphic vector bundles corresponding to Lagrangian submanifolds as in \cite{dg}, \cite{abouzaid}. Then, as expressed above, we can construct a triangulated category from this DG-category. For two stable vector bundles $E(\mathcal{L}_{\frac{a}{n}})$ and $E(\mathcal{L}_{\frac{b}{m}})$, we assume dim$\mathrm{Ext}^1(E(\mathcal{L}_{\frac{b}{m}}),E(\mathcal{L}_{\frac{a}{n}}))=1$. For a non-trivial morphism $\psi \in \mathrm{Ext}^1(E(\mathcal{L}_{\frac{b}{m}}),E(\mathcal{L}_{\frac{a}{n}}))$, there exists the short exact sequence
\begin{equation*}
\begin{CD}
0@>>>E(\mathcal{L}_{\frac{a}{n}})@>>>C(\psi )@>>>E(\mathcal{L}_{\frac{b}{m}})@>>>0.
\end{CD}
\end{equation*}
Here, $C(\psi )$ denotes the mapping cone of $\psi $, and we see that $C(\psi )$ is stable (see \cite{Abelian}). Then, we obtain the following exact triangle consisting of stable vector bundles, where $T$ is the shift functor in the triangulated category.
\begin{equation}
\begin{CD}
\cdots E(\mathcal{L}_{\frac{a}{n}})@>>>C(\psi )@>>>E(\mathcal{L}_{\frac{b}{m}})@>\psi >>TE(\mathcal{L}_{\frac{a}{n}})\cdots. \label{tori}
\end{CD}
\end{equation}
On the other hand, the isomorphism classes of indecomposable holomorphic vector bundles over an elliptic curve are classified by Atiyah \cite{atiyah}. This Atiyah's result implies that the set of isomorphism classes of indecomposable holomorphic vector bundles are parametrized by $\mu \in \check{T}^2=\mathbb{C}/2\pi (\mathbb{Z}\oplus \tau \mathbb{Z})$, where $\tau \in \mathbb{H}$. So we denote by $E(\mathcal{L}_{\frac{a}{n}})_{\mu }$ the representative of the isomorphism class [$E(\mathcal{L}_{\frac{a}{n}})_{\mu }$] of $E(\mathcal{L}_{\frac{a}{n}})$ corresponding to $\mu $. Note that this $E(\mathcal{L}_{\frac{a}{n}})_{\mu }$ corresponds to $E_{(\frac{a}{n},\frac{\mu }{n})}$ in the body of this paper. Now we turn to the case of the exact triangle $(\ref{tori})$. We see that $C(\psi )$ is a holomorphic vector bundle whose rank and degree are $m+n$ and $a+b$, respectively, so we expect that there exists a $\mu \in \check{T}^2$ such that $[C(\psi )]=[E(\mathcal{L}_{\frac{a+b}{m+n}})_{\mu }]$ by Atiyah's result. In this paper, we determine the value $\mu $ such that $[C(\psi )]=[E(\mathcal{L}_{\frac{a+b}{m+n}})_{\mu }]$ in the case dim$\mathrm{Ext}^1(E(\mathcal{L}_{\frac{b}{m}}),E(\mathcal{L}_{\frac{a}{n}}))=1$ as the most fundamental example of exact triangles such that $\psi \in \mathrm{Ext}^1(E(\mathcal{L}_{\frac{b}{m}}),E(\mathcal{L}_{\frac{a}{n}}))$ is non-trivial, and construct an isomorphism $C(\psi )\cong E(\mathcal{L}_{\frac{a+b}{m+n}})_{\mu }$ explicitly by employing theta functions. In particular, the Lagrangian submanifolds corresponding to $C(\psi )$ intersect at one point, and they become a single Lagrangian submanifold corresponding to $E(\mathcal{L}_{\frac{a+b}{m+n}})_{\mu }$ by an isomorphism $C(\psi )\cong E(\mathcal{L}_{\frac{a+b}{m+n}})_{\mu }$. We expect that this can be regarded as an analogue of the Dehn twist (see \cite{pic-lef}, \cite{abouzaid}). Furthermore, we give a geometric interpretation of the exact triangles consisting of stable vector bundles in terms of the corresponding Fukaya category via the homological mirror symmetry.

This paper is organized as follows. In section 2, we define holomorphic vector bundles associated to Lagrangian submanifolds. In section 3, we discuss the DG-category consisting of those holomorphic vector bundles, and comment the stability of them. In section 4, we discuss the structures of an exact triangle consisting of stable vector bundles $(\ref{tori})$ in the case dim$\mathrm{Ext}^1(E(\mathcal{L}_{\frac{b}{m}}),E(\mathcal{L}_{\frac{a}{n}}))=1$, where $\psi $ is non-trivial. Then, without loss of generality we may discuss the case $(n,a)=(1,0)$, $(m,b)=(1,1)$ only, because $E(\mathcal{L}_{\frac{a}{n}})$ is associated to $\mathcal{L}_{\frac{a}{n}}$, and $\mathcal{L}_{\frac{a}{n}}$ is transformed by the $SL(2;\mathbb{Z})$ action on $T^2$. We explain details of this fact in section 6. So in section 4, we consider two stable line bundles $E(\mathcal{L}_0)$ and $E(\mathcal{L}_1)$, and we take the mapping cone of $\psi : E(\mathcal{L}_1)\rightarrow TE(\mathcal{L}_0)$. Then $C(\psi )$ is a holomorphic vector bundle whose rank and degree are 2 and 1, respectively. Hence, we determine the value $\mu $ such that $[C(\psi )]=[E(\mathcal{L}_{\frac{1}{2}})_{\mu }]$, and in particular, for $E(\mathcal{L}_{\frac{1}{2}})_{\mu }$, compute non-trivial morphisms $\phi : C(\psi )\rightarrow E(\mathcal{L}_{\frac{1}{2}})_{\mu }$ and $\tilde{\phi } : E(\mathcal{L}_{\frac{1}{2}})_{\mu }\rightarrow C(\psi )$ explicitly. In fact, we can check $\phi \tilde{\phi } =c\cdot id_{E(\mathcal{L}_{\frac{1}{2}})_{\mu }}$ and $\tilde{\phi }\phi =c\cdot id_{C(\psi )}$ with a complex number $c\not=0$, and which implies $C(\psi )\cong E(\mathcal{L}_{\frac{1}{2}})_{\mu }$. In these arguments, theta functions play an important role. The value $\mu $ such that $[C(\psi )]=[E(\mathcal{L}_{\frac{1}{2}})_{\mu }]$ is given in Theorem \ref{the4.9}. In section 5, we discuss a geometric interpretation of the mapping cone $C(\psi )$ from the viewpoint of the corresponding symplectic geometry. In particular, for the isomorphisms $\phi $, $\tilde{\phi } $, we interpret the value $c$ as the structure constant of an $A_{\infty }$ product in the corresponding Fukaya category. In section 6, we explain the $SL(2;\mathbb{Z})$ action on $T^2$.

\section{Holomorphic vector bundles and Lagrangian submanifolds}
In this section, we define Lagrangian submanifolds on $T^2$, and define holomorphic vector bundles associated to those Lagrangian submanifolds on the dual torus $\check{T}^2$ of $T^2$. Here, $\check{T}^2$ is a complex torus. These are based on the SYZ construction \cite{syz} (see also \cite{Le}). 

First, we explain $\check{T}^2$. Let us denote the coordinates of the covering space $\mathbb{R}^2$ of $\check{T}^2$ by $(x,y)$. We also regard $(x,y)$ as a point of $\check{T}^2$ by identifying $x\sim x+2\pi $ and $y\sim y+2\pi $. We fix an $\varepsilon _0>0$ small enough and define
\begin{flalign*}
&O_{ij}:=\biggl\{(x,y)\in \check{T}^2 \ \mid \ \frac{2}{3}\pi (i -1)-\varepsilon _0<x<\frac{2}{3}\pi i +\varepsilon _0, &\\
&\hspace{53mm}\frac{2}{3}\pi (j -1)-\varepsilon _0<y<\frac{2}{3}\pi j +\varepsilon _0,\ \varepsilon _0>0\biggr\}, &
\end{flalign*}
where $i,j=1,2,3$. We can regard $O_{ij}$ as a open set of $\mathbb{R}^2$, and we define the local coordinates of $O_{ij}$ by $(x,y)\in \mathbb{R}^2$. Furthermore, we define the complex coordinate of $\check{T}^2$ as follows. Let $\tau $ be a complex number which satisfies Im$\tau>0$, and for $\check{T}^2$, we locally define the complex coordinate by $z=x+\tau y$. Namely, for the lattice generated by 1 and $\tau $, $\check{T}^2$ is isomorphic to $\mathbb{C}/2\pi (\mathbb{Z}\oplus \tau \mathbb{Z})$. 

On the other hand, we consider the dual symplectic torus $(T^2,\omega )$ whose complexified symplectic form $\omega $ is defined by
\begin{equation*}
\omega :=-\frac{1}{\tau }dx\wedge dy
\end{equation*}
as the mirror pair of $\check{T}^2\cong \mathbb{C}/2\pi (\mathbb{Z}\oplus \tau \mathbb{Z})$. We also denote the local coordinates of $(T^2,\omega )$ by the same notation $(x,y)$ since it may not cause any confusion. We define a map $s_{(\frac{a}{n},\frac{\mu }{n})} : \mathbb{R}\rightarrow \mathbb{C}$ by
\begin{equation*}
s_{(\frac{a}{n},\frac{\mu }{n})}(x):=\frac{a}{n}x+\frac{\mu }{n}.
\end{equation*}
Here, we assume $n\in \mathbb{N}$ and $a\in \mathbb{Z}$ are relatively prime and $\mu \in \mathbb{C}$. We denote $\mu =p+q\tau $ with $p$, $q\in \mathbb{R}$. Removing the term $\frac{q}{n}\tau $ from the above $s_{(\frac{a}{n},\frac{\mu }{n})}(x)$, the graph of 
\begin{equation*}
y=\frac{a}{n}x+\frac{p}{n}
\end{equation*}
defines a Lagrangian submanifold $\mathcal{L}_{(\frac{a}{n},\frac{p}{n})}$ in $\mathbb{R}^2$. Then, for the covering map $\pi : \mathbb{R}^2\rightarrow T^2$, the image $\pi (\mathcal{L}_{(\frac{a}{n},\frac{p}{n})})$ is a cycle which winds $n$ times in the base space direction and $a$ times in the fiber direction. This $\pi (\mathcal{L}_{(\frac{a}{n},\frac{p}{n})})$ is an example of a (special) Lagrangian submanifold in $(T^2,\omega )$. Hereafter, for simplicity, we denote by $\mathcal{L}_{(\frac{a}{n},\frac{p}{n})}$ a Lagrangian submanifold in $(T^2,\omega )$ instead of $\pi (\mathcal{L}_{(\frac{a}{n},\frac{p}{n})})$, too. Here, we explain the term $\frac{q}{n}\tau $ in the formula of $s_{(\frac{a}{n},\frac{\mu }{n})}$ briefly. In the homological mirror symmetry setting, we consider the Fukaya category in the symplectic geometry side (see section 5). Note that an object of the Fukaya category is a Lagrangian submanifold $\mathcal{L}$ endowed with a local system $(\mathscr{L},\nabla_{\mathscr{L}})$. Now, for $\mathcal{L}_{(\frac{a}{n},\frac{p}{n})}$, we consider a flat complex line bundle $\mathscr{L}_{(\frac{a}{n},\frac{\mu }{n})}\rightarrow \mathcal{L}_{(\frac{a}{n},\frac{p}{n})}$ whose $U(1)$-connection $\nabla_{\mathscr{L}_{(\frac{a}{n},\frac{\mu }{n})}}$ is defined by
\begin{equation*}
\nabla_{\mathscr{L}_{(\frac{a}{n},\frac{\mu }{n})}}:=d-\frac{\mathbf{i}}{2\pi } \frac{q}{n} dx,
\end{equation*}
where $\mathbf{i}=\sqrt{-1}$ and $d$ denotes the exterior derivative. Then, the number $q$ corresponds to the $U(1)$ holonomy of $(\mathscr{L}_{(\frac{a}{n},\frac{\mu }{n})},\nabla_{\mathscr{L}_{(\frac{a}{n},\frac{\mu }{n})}})$ along $\mathcal{L}_{(\frac{a}{n},\frac{p}{n})}\cong S^1$. In this sense, we call the number $q$ the $U(1)$ holonomy or simply the holonomy. Thus, giving a map $s_{(\frac{a}{n},\frac{\mu }{n})}$ is equivalent to giving an object $(\mathcal{L}_{(\frac{a}{n},\frac{p}{n})},\mathscr{L}_{(\frac{a}{n},\frac{\mu }{n})},\nabla_{\mathscr{L}_{(\frac{a}{n},\frac{\mu }{n})}})$ of the Fukaya category $\mathrm{Fuk}(T^2,\omega )$.

For $s_{(\frac{a}{n},\frac{\mu }{n})}$, we can associate the following holomorphic vecter bundle $E_{(\frac{a}{n},\frac{\mu }{n})}$ whose rank and degree are $n$ and $a$, respectively, on $\check{T}^2$. Let $U$ and $V$ be following square matrices of order $n$.
\begin{equation*}
U:=\left(\begin{array}{cccc}1&&&\\&\omega &&\\&&\ddots&\\&&&\omega ^{n-1}\end{array}\right),\ V:=\left(\begin{array}{cccc}0\ \ 1&&\\&\ddots\ddots&\\&&1\\1&&0\end{array}\right).
\end{equation*}
Here, $\omega $ is the $n$-th root of 1. Using these matrices, the transition functions of $E_{(\frac{a}{n},\frac{\mu }{n})}$ are defined as follows. Let $\psi _{(i,j)} : O_{ij} \rightarrow \mathbb{C}^n$ be a smooth section of $E_{(\frac{a}{n},\frac{\mu }{n})}$, where $i,j=1,2,3$. We define the transition function on $O_{3j}\cap O_{1j}$ by 
\begin{equation*}
\left.\psi _{(3,j)}\right|_{O_{3j}\cap O_{1j}}=\displaystyle{e^{\frac{a}{n}\mathbf{i}y}V\left.\psi _{(1,j)}\right|_{O_{3j}\cap O_{1j}}}.
\end{equation*}
Similarly, we define the transition function on $O_{i3}\cap O_{i1}$ by
\begin{equation*}
\left.\psi _{(i,3)}\right|_{O_{i3}\cap O_{i1}}=U^{-a}\left.\psi _{(i,1)}\right|_{O_{i3}\cap O_{i1}},
\end{equation*}
where $U^{-a}:=(U^{-1})^a$. Note that the transition functions which are defined on $O_{1j}\cap O_{2j}$, $O_{2j}\cap O_{3j}$, $O_{i1}\cap O_{i2}$ and $O_{i2}\cap O_{i3}$ are trivial. Moreover, when we define 
\begin{equation*}
\left.\psi _{(3,3)}\right|_{O_{33}\cap O_{11}}=U^{-a}\left.\psi _{(3,1)}\right|_{O_{33}\cap O_{11}}=\left(U^{-a}\right)\left(e^{\frac{a}{n}\mathbf{i}y}V\right)\left.\psi _{(1,1)}\right|_{O_{33}\cap O_{11}},
\end{equation*}
we can check that they satisfy the cocycle condition. We define a connection on $E_{(\frac{a}{n},\frac{\mu }{n})}$ locally as 
\begin{equation*}
D=d+A:=d-\frac{\mathbf{i}}{2\pi }s_{(\frac{a}{n},\frac{\mu }{n})}(x)dy\cdot I_n=d-\frac{\mathbf{i}}{2\pi }\left(\frac{a}{n}x+\frac{\mu }{n}\right)dy\cdot I_n,
\end{equation*}
where $I_n$ is the identity matrix of order $n$. In fact, $D$ is compatible with the transition functions and so defines a global connection. Strictly speaking, we should denote by $D_{(\frac{a}{n},\frac{\mu }{n})}$ a connection of $E_{(\frac{a}{n},\frac{\mu }{n})}$, but we denote by $D$ a connection of $E_{(\frac{a}{n},\frac{\mu }{n})}$ here, because we do not use the notation $D$ so much in this paper. Exceptionally, sometimes we denote $D_{\frac{a}{n}}$ instead of $D$ in section 3. Then its curvature form $F$ is 
\begin{equation*}
F=-\frac{\mathbf{i}}{2\pi }\frac{a}{n}dx\wedge dy\cdot I_n.
\end{equation*}
Since $\check{T}^2=\mathbb{C}/2\pi (\mathbb{Z}\oplus \tau \mathbb{Z})$ is a 1-dimensional complex manifold, the (0,2)-part of this curvature form vanishes automatically. Thus, for a complex vector bundle $E_{(\frac{a}{n},\frac{\mu }{n})}$ of rank $n$, $D$ defines the structure of a holomorphic vector bundle. In particular, the case $(n,a)=(1,0)$ with $\mu \in 2\pi n\mathbb{Z}$ corresponds to the trivial line bundle in this definition. Since
\begin{equation*}
\displaystyle{\int_{\check{T}^2}c_1(E_{(\frac{a}{n},\frac{\mu }{n})})=\frac{a}{4\pi ^2}\int_{\check{T}^2}dx\wedge dy=\frac{a}{4\pi ^2}4\pi ^2=a},
\end{equation*}
the degree of $E_{(\frac{a}{n},\frac{\mu }{n})}$ is $a$. Here, $c_1(E_{(\frac{a}{n},\frac{\mu }{n})})$ denotes the first chern class of $E_{(\frac{a}{n},\frac{\mu }{n})}$.
\begin{rem}
All diagonal components of $A$ are defined by $-\frac{\mathbf{i}}{2\pi }(\frac{a}{n}x+\frac{\mu }{n})dy$ in order to maintain the compatibility of the connection 1-form with respect to the transition functions. In fact, when we define the $(i,i)$ components of a connection 1-form $\tilde{A}\ (1\leq i\leq n)$ by $-\frac{\mathbf{i}}{2\pi }(\frac{a}{n}x+\frac{\mu_{ii} }{n})dy$, the relation $\mu _{11}=\cdots =\mu _{nn}$ holds by the condition such that $\tilde{A}$ is compatible with the transition functions.
\end{rem}
In general, a smooth section $\psi (x,y)$ of $E_{(\frac{a}{n},\frac{\mu }{n})}$ is expressed locally as follows. For each $x\in S^1$, $\psi (x,\cdot )$ gives a smooth function on the fiber $S^1$. Thus, $\psi (x,y)$ can be Fourier-expanded locally as 
\begin{equation*}
\psi (x,y)=\left(\begin{array}{ccc}\displaystyle{\sum_{I_1\in \mathbb{Z}}}\psi _{\lambda ,I_1}(x)e^{\frac{\mathbf{i}}{n}I_1y}\\\vdots\\\displaystyle{\sum_{I_n\in \mathbb{Z}}}\psi _{\lambda ,I_n}(x)e^{\frac{\mathbf{i}}{n}I_ny}\end{array}\right). 
\end{equation*}

\section{The DG-category consisting of holomorphic vector bundles}
In this section, we construct a DG-category $DG_{\check{T}^2}$ consisting of holomorphic vector bundles defined in section 2. This is an extension of the DG-category of holomorphic line bundles in \cite{line}. The objects of $DG_{\check{T}^2}$ are holomorphic vector bundles $E_{(\frac{a}{n},\frac{\mu }{n})}$ with $U(n)$-connections $D$. Hereafter sometimes we denote $D_{\frac{a}{n}}$ instead of $D$ in order to specify that $D$ is associated to $E_{(\frac{a}{n},\frac{\mu }{n})}$. We often label these objects as $s_{(\frac{a}{n},\frac{\mu }{n})}$ instead of $(E_{(\frac{a}{n},\frac{\mu }{n})},D_{\frac{a}{n}})$, because $E_{(\frac{a}{n},\frac{\mu }{n})}$ is associated to $s_{(\frac{a}{n},\frac{\mu }{n})}$. For any two objects $s_{(\frac{a}{n},\frac{\mu }{n})}=(E_{(\frac{a}{n},\frac{\mu }{n})},D_{\frac{a}{n}})$, $s_{(\frac{b}{m},\frac{\nu }{m})}=(E_{(\frac{b}{m},\frac{\nu }{m})},D_{\frac{b}{m}})\in Ob(DG_{\check{T}^2})$, the space $DG_{\check{T}^2}(s_{(\frac{a}{n},\frac{\mu }{n})},s_{(\frac{b}{m},\frac{\nu }{m})})$ of morphisms is defined by
\begin{equation*}
DG_{\check{T}^2}(s_{(\frac{a}{n},\frac{\mu }{n})},s_{(\frac{b}{m},\frac{\nu }{m})}):=\Gamma (E_{(\frac{a}{n},\frac{\mu }{n})},E_{(\frac{b}{m},\frac{\nu }{m})})\otimes _{C^{\infty}(\check{T}^2)}\Omega ^{0,\ast }(\check{T}^2),
\end{equation*}
where $\Omega ^{0,\ast }(\check{T}^2)$ is the space of anti-holomorphic differential forms, and $\Gamma (E_{(\frac{a}{n},\frac{\mu }{n})},$ \\ $E_{(\frac{b}{m},\frac{\nu }{m})})$ is the space of homomorphisms from $E_{(\frac{a}{n},\frac{\mu }{n})}$ to $E_{(\frac{b}{m},\frac{\nu }{m})}$. The space $DG_{\check{T}^2}(s_{(\frac{a}{n},\frac{\mu }{n})},s_{(\frac{b}{m},\frac{\nu }{m})})$ is a $\mathbb{Z}$-graded vector space, where the grading is defined as the degree of the anti-holomorphic differential forms. The degree $r$ part is denoted $DG^r_{\check{T}^2}(s_{(\frac{a}{n},\frac{\mu }{n})},s_{(\frac{b}{m},\frac{\nu }{m})})$. We define a linear map $d_{(\frac{a}{n},\frac{b}{m})} : DG^r_{\check{T}^2}(s_{(\frac{a}{n},\frac{\mu }{n})},s_{(\frac{b}{m},\frac{\nu }{m})})\rightarrow DG^{r+1}_{\check{T}^2}(s_{(\frac{a}{n},\frac{\mu }{n})},s_{(\frac{b}{m},\frac{\nu }{m})})$ as follows. We decompose $D_{\frac{a}{n}}$ into its holomorphic part and anti-holomorphic part $D_{\frac{a}{n}}=D^{(1,0)}_{\frac{a}{n}}+D^{(0,1)}_{\frac{a}{n}}$, and set $d_{\frac{a}{n}}:=2D^{(0,1)}_{\frac{a}{n}}$. Then, for $\psi \in DG^r_{\check{T}^2}(s_{(\frac{a}{n},\frac{\mu }{n})},s_{(\frac{b}{m},\frac{\nu }{m})})$, we set 
\begin{equation*}
d_{(\frac{a}{n},\frac{b}{m})}(\psi ):=d_{\frac{b}{m}}\psi -(-1)^r\psi d_{\frac{a}{n}}.
\end{equation*}
Furthermore, for any $\psi \in DG^r_{\check{T}^2}(s_{(\frac{a}{n},\frac{\mu }{n})},s_{(\frac{b}{m},\frac{\nu }{m})})$, $d^2_{(\frac{a}{n},\frac{b}{m})}(\psi )$ is expressed as follows.
\begin{align*}
d^2_{(\frac{a}{n},\frac{b}{m})}(\psi )&=d_{(\frac{a}{n},\frac{b}{m})}(d_{\frac{b}{m}}\psi -(-1)^r\psi d_{\frac{a}{n}})\\
                                             &=d^2_{\frac{b}{m}}\psi -\psi d^2_{\frac{a}{n}}.
\end{align*}
Here, $d^2_{\frac{a}{n}}=0$ because $d_{\frac{a}{n}}=2D^{(0,1)}_{\frac{a}{n}}$ and $E_{(\frac{a}{n},\frac{\mu }{n})}$ is holomorphic. Similarly one has $d_{\frac{b}{m}}^2=0$. Thus $d^2_{(\frac{a}{n},\frac{b}{m})}(\psi )=0$, so $d_{(\frac{a}{n},\frac{b}{m})}$ defines a differential. Sometimes we denote by $H^r(E_{(\frac{a}{n},\frac{\mu }{n})},E_{(\frac{b}{m},\frac{\nu }{m})})$ the $r$-th cohomology with respect to the differential $d_{(\frac{a}{n},\frac{b}{m})}$, and in particular, $H^0(E_{(\frac{a}{n},\frac{\mu }{n})},E_{(\frac{b}{m},\frac{\nu }{m})})$ is the space of holomorphic maps. Furthermore, when $s_{(\frac{a}{n},\frac{\mu }{n})}$ is the zero section, $s_{(\frac{a}{n},\frac{\mu }{n})}=s_{(0,0)}$, the differential $d_{(0,\frac{b}{m})} : DG^r_{\check{T}^2}(s_{(0,0)},s_{(\frac{b}{m},\frac{\nu }{m})})\rightarrow DG^{r+1}_{\check{T}^2}(s_{(0,0)},s_{(\frac{b}{m},\frac{\nu }{m})})$ is also denoted $d_{\frac{b}{m}}:=d_{(0,\frac{b}{m})}$. The product structure $m : DG_{\check{T}^2}(s_{(\frac{b}{m},\frac{\nu }{m})},s_{(\frac{c}{l},\frac{\eta }{l})})\otimes DG_{\check{T}^2}(s_{(\frac{a}{n},\frac{\mu }{n})},s_{(\frac{b}{m},\frac{\nu }{m})})\rightarrow DG_{\check{T}^2}(s_{(\frac{a}{n},\frac{\mu }{n})},s_{(\frac{c}{l},\frac{\eta }{l})})$ is defined by the composition of homomorphisms of vector bundles together with the wedge product for the anti-holomorphic differential forms. We can check that $d_{(\frac{a}{n},\frac{b}{m})}$ and $m$ satisfy the Leibniz rule. Thus, $DG_{\check{T}^2}$ forms a DG-category.

For later convenience, we give the local expression of $d_{(\frac{a}{n},\frac{b}{m})}$ explicitly. Since $z=x+\tau y$ and $\bar{z}=x+\bar{\tau }y$, one has
\begin{equation*}
dx=-\frac{\bar{\tau }}{\tau -\bar{\tau }}dz+\frac{\tau }{\tau -\bar{\tau }}d\bar{z},\ \ dy=\frac{1}{\tau -\bar{\tau }}(dz-d\bar{z}).
\end{equation*}
Using these, we decompose $D_{\frac{a}{n}}=d-\frac{\mathbf{i}}{2\pi }s_{(\frac{a}{n},\frac{\mu }{n})}(x)dy\cdot I_n$ into its holomorphic part and anti-holomorphic part.
\begin{align*}
D_{\frac{a}{n}}&=d-\frac{\mathbf{i}}{2\pi }s_{(\frac{a}{n},\frac{\mu }{n})}(x)dy\cdot I_n\\
&=\partial +\frac{\mathbf{i}}{2\pi (\bar{\tau } -\tau)} s_{(\frac{a}{n},\frac{\mu }{n})}(x)dz\cdot I_n+\bar{\partial }-\frac{\mathbf{i}}{2\pi (\bar{\tau } -\tau)} s_{(\frac{a}{n},\frac{\mu }{n})}(x)d\bar{z}\cdot I_n.
\end{align*}
Since $d_{\frac{a}{n}}=2D^{(0,1)}_{\frac{a}{n}}$, so $d_{\frac{a}{n}}=2\bar{\partial }-\frac{\mathbf{i}}{\pi (\bar{\tau }-\tau )}s_{(\frac{a}{n},\frac{\mu }{n})}(x)d\bar{z}\cdot I_n$. Thus, for any $\psi \in DG^r_{\check{T}^2}(s_{(\frac{a}{n},\frac{\mu }{n})},s_{(\frac{b}{m},\frac{\nu }{m})})$, $d_{(\frac{a}{n},\frac{b}{m})}(\psi )$ is expressed locally as 
\begin{equation*}
d_{(\frac{a}{n},\frac{b}{m})}(\psi )=2\bar{\partial }(\psi ) -\frac{\mathbf{i}}{\pi (\bar{\tau }-\tau )}(s_{(\frac{b}{m},\frac{\nu }{m})}(x)-s_{(\frac{a}{n},\frac{\mu }{n})}(x))d\bar{z} \cdot I_n\wedge \psi.
\end{equation*}

On the other hand, for two objects $s_{(\frac{a}{n},\frac{\mu }{n})},s_{(\frac{b}{m},\frac{\nu }{m})}\in Ob(DG_{\check{T}^2})$, $DG^0_{\check{T}^2}(s_{(\frac{a}{n},\frac{\mu }{n})},$ \\ 
$s_{(\frac{b}{m},\frac{\nu }{m})})$ is the space of sections of $E_{(\frac{b}{m},\frac{\nu }{m} )}$ if $s_{(\frac{a}{n},\frac{\mu }{n})}$ is the zero section, $s_{(\frac{a}{n},\frac{\mu }{n})}=s_{(0,0)}$. Furthermore, in section 2, we saw that smooth sections of $E_{(\frac{a}{n},\frac{\mu }{n})}$ can be Fourier-expanded locally. In fact, any morphism $\psi \in DG^r_{\check{T}^2}(s_{(\frac{a}{n},\frac{\mu }{n})},s_{(\frac{b}{m},\frac{\nu }{m})})$ can be Fourier-expanded locally in a similar way.

Recall that $H^0(E_{(\frac{a}{n},\frac{\mu }{n})},E_{(\frac{b}{m},\frac{\nu }{m})})$ is the space of holomorphic maps. For two holomorphic vector bundles whose ranks and degrees are same, i.e., $(n,a)=(m,b)$, the following proposition holds.
\begin{proposition}\label{pro3.1}
Let n be a natural number and a an integer. We assume n and a are relatively prime. Then for $\mu $ and $\nu \ (\mu ,\nu \in \mathbb{C})$, $\mathrm{dim}H^0(E_{(\frac{a}{n},\frac{\mu }{n})},E_{(\frac{a}{n},\frac{\nu }{n})})$ \\
$=1$ if and only if $\mu \equiv \nu \ (\mathrm{mod}\ 2\pi (\mathbb{Z} \oplus \tau \mathbb{Z}))$, and, $\mathrm{dim}H^0(E_{(\frac{a}{n},\frac{\mu }{n})},E_{(\frac{a}{n},\frac{\nu }{n})})=0$ otherwise. Furthermore, in the case of $\mathrm{dim}H^0(E_{(\frac{a}{n},\frac{\mu }{n})},E_{(\frac{a}{n},\frac{\nu }{n})})=1$, a non-trivial element in $H^0(E_{(\frac{a}{n},\frac{\mu }{n})},E_{(\frac{a}{n},\frac{\nu }{n})})$ gives an isomorphism $E_{(\frac{a}{n},\frac{\mu }{n})}\cong E_{(\frac{a}{n},\frac{\nu }{n})}$.
\end{proposition}
\begin{proof}
We look for a non-trivial holomorphic map $\Phi \in H^0(E_{(\frac{a}{n},\frac{\mu }{n})},E_{(\frac{a}{n},\frac{\nu }{n})})$. By definition, rank$E_{(\frac{a}{n},\frac{\mu }{n})}$=rank$E_{(\frac{a}{n},\frac{\nu }{n})}$=$n$, so $\Phi $ is expressed locally as 
\begin{equation*}
\Phi =\left(\begin{array}{ccc}\Phi _{11}(x,y)\cdots\Phi _{1n}(x,y)\\\vdots \ \ \ \ \ \ \ddots \ \ \ \ \ \ \vdots\\\Phi _{n1}(x,y)\cdots\Phi _{nn}(x,y)\end{array}\right).
\end{equation*}
As discussed in section 2, the transition function for a smooth section $\psi _{(i,j)} : O_{ij} \rightarrow \mathbb{C}^n$ ($i,j=1,2,3$) of $E_{(\frac{a}{n},\frac{\mu }{n})}$ is defined by 
\begin{equation*}
\left.\psi _{(3,j)}\right|_{O_{3j}\cap O_{1j}}=\displaystyle{e^{\frac{a}{n}\mathbf{i}y}V\left.\psi _{(1,j)}\right|_{O_{3j}\cap O_{1j}}},\ \left.\psi _{(i,3)}\right|_{O_{i3}\cap O_{i1}}=U^{-a}\left.\psi _{(i,1)}\right|_{O_{i3}\cap O_{i1}},
\end{equation*}
where the transition functions on $O_{1j}\cap O_{2j}$, $O_{2j}\cap O_{3j}$, $O_{i1}\cap O_{i2}$ and $O_{i2}\cap O_{i3}$ are trivial. The same thing holds also for $E_{(\frac{a}{n},\frac{\nu }{n})}$. Then $\Phi _{ij}(x,y)\ (1\leq i,j\leq n)$ can be Fourier-expanded as 
\begin{equation*}
\Phi _{ij}(x,y)=e^{\frac{j-i}{n}a\mathbf{i}y}\displaystyle{\sum_{I_{ij}\in \mathbb{Z}}}\Phi _{ij,I_{ij}}(x)e^{\mathbf{i}I_{ij}y},
\end{equation*}
by considering the transition functions of $E_{(\frac{a}{n},\frac{\mu }{n})}$ and $E_{(\frac{a}{n},\frac{\nu }{n})}$ in the $y$ direction. Here, we consider the holomorphic structures of $E_{(\frac{a}{n},\frac{\mu }{n})}$ and $E_{(\frac{a}{n},\frac{\nu }{n})}$. The holomorphic structure of $E_{(\frac{a}{n},\frac{\mu }{n})}$ is defined by
\begin{equation*}
d_{\frac{a}{n}}:=\left(2\bar{\partial } -\frac{\mathbf{i}}{\pi (\bar{\tau } -\tau )}\left(\frac{a}{n}x+\frac{\mu }{n}\right)d\bar{z} \right)\cdot I_n,
\end{equation*}
and similarly, the holomorphic structure of $E_{(\frac{a}{n},\frac{\nu }{n})}$ is defined by
\begin{equation*}
d'_{\frac{a}{n}}:=\left(2\bar{\partial } -\frac{\mathbf{i}}{\pi (\bar{\tau } -\tau )}\left(\frac{a}{n}x+\frac{\nu }{n}\right)d\bar{z} \right)\cdot I_n.
\end{equation*}
Since $\Phi \in H^0(E_{(\frac{a}{n},\frac{\mu }{n})},E_{(\frac{a}{n},\frac{\nu }{n})})=\mathrm{Ker}\{d_{(\frac{a}{n},\frac{a}{n})} : DG^0_{\check{T}^2}(s_{(\frac{a}{n},\frac{\mu }{n})},s_{(\frac{a}{n},\frac{\nu }{n})})\rightarrow DG^1_{\check{T}^2}$\\
$(s_{(\frac{a}{n},\frac{\mu }{n})},s_{(\frac{a}{n},\frac{\nu }{n})})\}$, $\Phi $ satisfies the differential equation $d_{(\frac{a}{n},\frac{a}{n})}(\Phi )=d'_{\frac{a}{n}}\Phi -\Phi d_{\frac{a}{n}}=0$. Hence $\Phi _{ij}(x,y)\ (1\leq i,j\leq n)$ is given by 
\begin{equation*}
\Phi _{ij}(x,y)=e^{\frac{j-i}{n}a\mathbf{i}y}\displaystyle{\sum_{I_{ij}\in \mathbb{Z}}}C_{ij,I_{ij}}e^{-\frac{\mathbf{i}}{\pi n\tau }(\frac{\nu }{2}-\frac{\mu }{2}-\pi nI_{ij}-\pi a(j-i))x}e^{\mathbf{i}I_{ij}y},
\end{equation*}
where $C_{ij,I_{ij}}$ is an arbitrary constant. On the other hand, we obtain the following relations for any $i,j\ (1\leq i,j\leq n)$,
\begin{equation*}
\Phi _{ij}(x+2\pi ,y)=\Phi _{(i+1)(j+1)}(x,y),
\end{equation*}
by considering the transition functions of $E_{(\frac{a}{n},\frac{\mu }{n})}$ and $E_{(\frac{a}{n},\frac{\nu }{n})}$ in the $x$ direction. Namely, for each $l=1,\cdots,n$, the following relations hold.
\begin{align*}
&\Phi _{1l}(x+2\pi ,y)=\Phi _{2(l+1)}(x,y),\Phi _{2(l+1)}(x+2\pi ,y)=\Phi _{3(l+2)}(x,y),\cdots,\\
&\Phi _{(n-l)(n-1)}(x+2\pi ,y)=\Phi _{(n-l+1)n}(x,y),\Phi _{(n-l+1)n}(x+2\pi ,y)=\Phi _{(n-l+2)1}(x,y),\\
&\Phi _{(n-l+2)1}(x+2\pi ,y)=\Phi _{(n-l+3)2}(x,y),\cdots,\Phi _{(n-1)(l-2)}(x+2\pi ,y)=\Phi _{n(l-1)}(x,y),\\
&\Phi _{n(l-1)}(x+2\pi ,y)=\Phi _{1l}(x,y).
\end{align*}
These relations imply that the arbitrary constants satisfy the following relations.
\begin{align*}
&C_{1l,I_{1l}}e^{-\frac{\mathbf{i}}{n\tau }(\nu -\mu -2\pi nI_{1l}-2\pi a(l-1))}=C_{2(l+1),I_{1l}},\\
&C_{2(l+1),I_{2(l+1)}}e^{-\frac{\mathbf{i}}{n\tau }(\nu -\mu -2\pi nI_{2(l+1)}-2\pi a(l-1))}=C_{3(l+2),I_{2(l+1)}},\cdots,\\
&C_{(n-l)(n-1),I_{(n-l)(n-1)}}e^{-\frac{\mathbf{i}}{n\tau }(\nu -\mu -2\pi nI_{(n-l)(n-1)}-2\pi a(l-1))}=C_{(n-l+1)n,I_{(n-l)(n-1)}},\\
&C_{(n-l+1)n,I_{(n-l+1)n}}e^{-\frac{\mathbf{i}}{n\tau }(\nu -\mu -2\pi nI_{(n-l+1)n}-2\pi a(l-1))}=C_{(n-l+2)1,I_{(n-l+1)n}+a},\\
&C_{(n-l+2)1,I_{(n-l+2)1}+a}e^{-\frac{\mathbf{i}}{n\tau }(\nu -\mu -2\pi nI_{(n-l+2)1}-2\pi a(l-1))}=C_{(n-l+3)2,I_{(n-l+2)1}+a},\cdots,\\
&C_{(n-1)(l-2),I_{(n-l)(l-2)}+a}e^{-\frac{\mathbf{i}}{n\tau }(\nu -\mu -2\pi nI_{(n-l)(l-2)}-2\pi a(l-1))}=C_{n(l-1),I_{(n-l)(l-2)}+a},\\
&C_{n(l-1),I_{n(l-1)}+a}e^{-\frac{\mathbf{i}}{n\tau }(\nu -\mu -2\pi nI_{n(l-1)}-2\pi a(l-1))}=C_{1l,I_{n(l-1)}}.
\end{align*}
Hence we obtain 
\begin{equation*}
C_{1l,I_{1l}}=e^{\frac{\mathbf{i}}{\tau }(\nu -\mu -2\pi nI_{1l}-2\pi a(l-1))}C_{1l,I_{1l}}.
\end{equation*}
Here, if $e^{\frac{\mathbf{i}}{\tau }(\nu -\mu -2\pi nI_{1l}-2\pi a(l-1))}\not=1$ then $C_{1l,I_{1l}}=0\ (I_{1l}\in \mathbb{Z})$, so we consider the case $e^{\frac{\mathbf{i}}{\tau }(\nu -\mu -2\pi nI_{1l}-2\pi a(l-1))}=1$. We set that $\mu =p+q\tau $ and $\nu =p'+q'\tau \ (p,p',q,q'\in \mathbb{R})$, and one has
\begin{equation*}
1=e^{\frac{\mathbf{i}}{\tau }(\nu -\mu -2\pi nI_{1l}-2\pi a(l-1))}=e^{\frac{\mathbf{i}}{\tau }(p'-p-2\pi nI_{1l}-2\pi a(l-1))}e^{\mathbf{i}(q'-q)}.
\end{equation*}
Therefore $p,\ p',\ q,\ q'$ must satisfy the conditions $p'-p=2\pi (nI_{1l}+a(l-1))$, $q\equiv q'\ (\mathrm{mod}\ 2\pi \mathbb{Z})$. In particular, if we fix a value $l$ (clearly, there are $n$ ways how to fix a value $l$), we see that $\ C_{1l,I_{1l}}=c\in \mathbb{C}\ (c\not=0)$ for an $I_{1l}\in \mathbb{Z}$, and $C_{1l,I}=0$ for any $I\not=I_{1l}\ (I\in \mathbb{Z})$. Thus, when we set $2\pi k:=q'-q\ (k\in \mathbb{Z})$, 
\begin{align*}
&C_{2(l+1),I_{1l}}=c\omega ^{-k},\cdots,C_{(n-l+1)n,I_{1l}}=c\omega ^{-(n-l)k},\\
&C_{(n-l+2)1,I_{1l}+a}=c\omega ^{-(n-l+1)k},\cdots,C_{n(l-1),I_{1l}+a}=c\omega ^{-(n-1)k},
\end{align*}
and $C_{2(l+1),I}=\cdots=C_{(n-l+1)n,I}=C_{(n-l+2)1,I+a}=\cdots=C_{n(l-1),I+a}=0$ for any $I\not=I_{1l}\ (I\in \mathbb{Z})$. Here, $n$ and $a$ are relatively prime and $I_{1l}\in \mathbb{Z}$, so $p\equiv p'\ (\mathrm{mod}\ 2\pi \mathbb{Z})$ holds. Furthermore, all other components of $\Phi $ are zero by the conditions of transition functions in the $x$ direction for the values $l'\not=l,\ l'=1,\cdots,n$. Thus, for $\mu $ and $\nu $ ($\mu $, $\nu \in \mathbb{C}$), the condition $E_{(\frac{a}{n},\frac{\mu }{n})}\cong E_{(\frac{a}{n},\frac{\nu }{n})}$ is equivalent to the condition $\mu \equiv \nu \ (\mathrm{mod}\ 2\pi (\mathbb{Z}\oplus \tau \mathbb{Z}))$, and when $p'-p=2\pi (nI_{1l}+a(l-1))\ (I_{1l}\in \mathbb{Z})$ and $q'-q=2\pi k\ (k\in \mathbb{Z})$, the non-trivial holomorphic map $\Phi : E_{(\frac{a}{n},\frac{\mu }{n})}\rightarrow E_{(\frac{a}{n},\frac{\nu }{n})}$ is expressed locally as
\begin{align*}
&\Phi =e^{-\frac{k}{n}\mathbf{i}x+\mathbf{i}I_{1l}y+\frac{l-1}{n}a\mathbf{i}y}\left(\begin{array}{ccc}0&\Phi _1\\\Phi _2&0\end{array}\right),\\
&\Phi _1:=\left(\begin{array}{ccc}c&&\\&\ddots&\\&&c\omega ^{-(n-l)k}\end{array}\right),\ \Phi _2:=\left(\begin{array}{ccc}c\omega ^{-(n-l+1)k}&&\\&\ddots&\\&&c\omega ^{-(n-1)k}\end{array}\right),
\end{align*}
where $\Phi _1$ and $\Phi _2$ are square matrices of order $n-l+1$ and $l-1$, respectively ($c\in \mathbb{C}$, $c\not=0$). Since $\Phi $ has its inverse, it is an isomorphism.
\end{proof}
Thus, the isomorphism $\Phi : E_{(\frac{a}{n},\frac{\mu }{n})}\rightarrow E_{(\frac{a}{n},\frac{\nu }{n})}$ belongs to $H^0(E_{(\frac{a}{n},\frac{\mu }{n})},E_{(\frac{a}{n},\frac{\nu }{n})})$. By Proposition \ref{pro3.1}, we obtain the following corollary.
\begin{corollary}\label{cor3.2}
For $E_{(\frac{a}{n},\frac{\mu }{n})}$, a holomorphic map $\Phi : E_{(\frac{a}{n},\frac{\mu }{n})}\rightarrow E_{(\frac{a}{n},\frac{\mu }{n})}$ is expressed locally as $\Phi =cI_n$, where $c\in \mathbb{C}$.
\end{corollary}
\begin{proof}
We consider in the Proposition \ref{pro3.1} in the case $\mu =\nu $, where $l=1$, $I_{11}=0$ and $k=0$.
\end{proof}
Hence $E_{(\frac{a}{n},\frac{\mu }{n})}$ is simple, so it is indecomposable. In fact, it is known that an indecomposable vector bundle $E$ on an elliptic curve is stable if and only if the rank of $E$ and the degree of $E$ are relatively prime (see \cite{Abelian}, $\mathrm{p}$.178). Thus $E_{(\frac{a}{n},\frac{\mu }{n})}$ is stable.

Moreover, for $E_{(\frac{a}{n},\frac{\mu }{n})}$ and $E_{(\frac{b}{m},\frac{\nu }{m})}$ with $(n,a)\not=(m,b)$, the following proposition is known (see \cite{Abelian}, $\mathrm{p}$.179).
\begin{proposition}\label{pro3.3}
For $E_{(\frac{a}{n},\frac{\mu }{n})}$ and $E_{(\frac{b}{m},\frac{\nu }{m})}$, if $bn-am>0$ then $\mathrm{dim}H^0(E_{(\frac{a}{n},\frac{\mu }{n})},$\\
$E_{(\frac{b}{m},\frac{\nu }{m})})=bn-am$ and $\mathrm{dim}H^1(E_{(\frac{a}{n},\frac{\mu }{n})},E_{(\frac{b}{m},\frac{\nu }{m})})=0$, if $bn-am<0$ then $\mathrm{dim}H^0(E_{(\frac{a}{n},\frac{\mu }{n})},E_{(\frac{b}{m},\frac{\nu }{m})})=0$ and $\mathrm{dim}H^1(E_{(\frac{a}{n},\frac{\mu }{n})},E_{(\frac{b}{m},\frac{\nu }{m})})=am-bn$. 
\end{proposition}
In Proposition 3.3, if $bn-am>0$, then we see dim$H^0(E_{(\frac{a}{n},\frac{\mu }{n})},E_{(\frac{b}{m},\frac{\nu }{m})})=bn-am$ by a direct calculation, so one has dim$H^1(E_{(\frac{a}{n},\frac{\mu }{n})},E_{(\frac{b}{m},\frac{\nu }{m})})=0$ by Riemann-Roch theorem. We can also prove in the case of $bn-am<0$ similarly. The arguments of Proposition 3.3 correspond to the discussions of slope stability for $E_{(\frac{a}{n},\frac{\mu }{n})}$ and $E_{(\frac{b}{m},\frac{\nu }{m})}$ (see $\cite{Abelian}$). By Proposition 3.3, for $E_{(\frac{a}{n},\frac{\mu }{n})}$ and $E_{(\frac{b}{m},\frac{\nu }{m})}$ with $(n,a)\not=(m,b)$, either $H^0(E_{(\frac{a}{n},\frac{\mu }{n})},E_{(\frac{b}{m},\frac{\nu }{m})})$ or $H^0(E_{(\frac{b}{m},\frac{\nu }{m})},E_{(\frac{a}{n},\frac{\mu }{n})})$ is zero. Thus, when $(n,a)\not=(m,b)$, $E_{(\frac{a}{n},\frac{\mu }{n})}$ is not isomorphic to $E_{(\frac{b}{m},\frac{\nu }{m})}$.

\section{The construction of the isomorphism}
In this section, we construct the mapping cone of a morphism between holomorphic vector bundles on $\check{T}^2$, and discuss the structures of an exact triangle associated to the mapping cone  
\begin{equation*}
\begin{CD}
\cdots E_{(\frac{a}{n},\frac{\mu }{n})}@>>>C([\psi] )@>>>E_{(\frac{b}{m},\frac{\nu }{m})}@>[\psi ]>>TE_{(\frac{a}{n},\frac{\mu }{n})} \cdots
\end{CD}
\end{equation*}
in the case dim$\mathrm{Ext}^1(E_{(\frac{b}{m},\frac{\nu }{m})},E_{(\frac{a}{n},\frac{\mu }{n})})=1$, where $\frac{a}{n}<\frac{b}{m}$. Here, for a morphism $\psi \in DG_{\check{T}^2}^1(s_{(\frac{b}{m},\frac{\nu }{m})},s_{(\frac{a}{n},\frac{\mu }{n})})$, the cohomology class $[\psi ]\in \mathrm{Ext}^1(E_{(\frac{b}{m},\frac{\nu }{m})},E_{(\frac{a}{n},\frac{\mu }{n})})=H^1(E_{(\frac{b}{m},\frac{\nu }{m})},E_{(\frac{a}{n},\frac{\mu }{n})})$ is non-trivial, and $C([\psi ])$ denotes the mapping cone of $[\psi ]$. Hereafter, we also denote by $\psi $ a cohomology class of $\psi $ instead of $[\psi ]$. Then, without loss of generality we may discuss the case $(n,a)=(1,0)$, $(m,b)=(1,1)$ only, because we can consider the $SL(2;\mathbb{Z})$ action on $(T^2,\omega )$. As discussed in section 2, $E_{(\frac{a}{n},\frac{\mu }{n})}$ is associated to $s_{(\frac{a}{n},\frac{\mu }{n})}$, and $s_{(\frac{a}{n},\frac{\mu }{n})}$ is transformed by the $SL(2;\mathbb{Z})$ action on $(T^2,\omega )$. We explain this fact in section 6. Thus, we consider the mapping cone of $\psi =\tilde{\psi } d\bar{z}\in DG^1_{\check{T}^2}(s_{(1,\nu )},s_{(0,\mu )})$, where $\mu =p+q\tau$, $\nu =s+t\tau \ (p,q,s,t\in \mathbb{R})$. First, we recall the Atiyah's result on the classification of the isomorphism classes of indecomposable holomorphic vector bundles over an elliptic curve. 
\begin{theo}[Atiyah, 1957, \cite{atiyah}]\label{theati}
The set of isomorphism classes of indecomposable holomorphic vector bundles over an elliptic curve can be identified with the elliptic curve when the rank and degree of holmorphic vector bundles are relatively prime.
\end{theo}
We explain how to apply this Theorem \ref{theati} to our discussions. Since $E_{(\frac{a}{n},\frac{\mu }{n})}\cong E_{(\frac{a}{n},\frac{\nu }{n})}$ holds if and only if $\mu \equiv \nu \ (\mathrm{mod}\ 2\pi (\mathbb{Z}\oplus \tau \mathbb{Z}))$ by Proposition \ref{pro3.1}, the set of isomorphism classes of $E_{(\frac{a}{n},\frac{\mu }{n})}$ is parametrized by $\mu \in \mathbb{C}/2\pi (\mathbb{Z}\oplus \tau \mathbb{Z})$. Now $C(\psi )$ is a holomorphic vector bundle whose rank and degree are 2 and 1, respectively. So if $C(\psi )$ is indecomposable, we expect that there exists an $\eta \in \mathbb{C}$ such that $C(\psi )\cong E_{(\frac{1}{2},\frac{\eta }{2})}$ by Theorem \ref{theati}. In fact, for $E_{(\frac{1}{2},\frac{\eta }{2})}$, where $\eta =u+v\tau \ (u,v\in \mathbb{R})$, $C(\psi )\cong E_{(\frac{1}{2},\frac{\eta }{2})}$ holds if and only if $\eta \equiv \mu +\nu +\pi +\pi \tau \ (\mathrm{mod}\ 2\pi (\mathbb{Z}\oplus \tau \mathbb{Z}))$ (Theorem \ref{the4.9}). This is our main theorem which we will show in this section.

Generally, for a given DG-category, we can construct a DG-category consisting of one-sided twisted complexes from the original DG-category, and obtain a triangulated category as the 0-th cohomology of the DG-category of one-sided twisted complexes \cite{bondal}. Here, we denote by $Tr(DG_{\check{T}^2})$ the triangulated category obtained by this construction from $DG_{\check{T}^2}$. In $Tr(DG_{\check{T}^2})$, there exists the following exact triangle associated to the mapping cone, 
\begin{equation*}
\begin{CD}
\cdots T^{-1}E_{(1,\nu )}@>T^{-1}\psi >>E_{(0,\mu )}@>\iota >>C(\psi )@>\pi >>E_{(1,\nu )}@>\psi >>TE_{(0,\mu )}\cdots,
\end{CD}
\end{equation*}
where $T$ is the shift functor. We discuss the conditions when there exist non-trivial holomorphic maps such that $\tilde{\phi } : E_{(\frac{1}{2},\frac{\eta }{2})}\rightarrow C(\psi )\ (\tilde{\phi } \not=0)$ and $\phi : C(\psi )\rightarrow E_{(\frac{1}{2},\frac{\eta }{2})}\ (\phi \not=0)$. We apply the covariant cohomological functor $F:=\mathrm{Hom}(E_{(\frac{1}{2},\frac{\eta }{2})},\cdot )$ and contravariant cohomological functor $G:=\mathrm{Hom}(\cdot ,E_{(\frac{1}{2},\frac{\eta }{2})})$ to it, and obtain the following long exact sequence.
\begin{equation*}
\begin{CD}
\cdots F(E_{(0,\mu )})@>F(\iota )>>F(C(\psi ))@>F(\pi )>>F(E_{(1,\nu )})@>F(\psi )>>F(TE_{(0,\mu )})\cdots,
\end{CD}
\end{equation*}
\begin{equation*}
\begin{CD}
\cdots G(E_{(1,\nu )})@>G(\pi )>>G(C(\psi ))@>G(\iota )>>G(E_{(0,\mu )})@>G(T^{-1}\psi )>>G(T^{-1}E_{(1,\nu )})\cdots.
\end{CD}
\end{equation*}
Then we obtain the following lemma. 
\begin{lemma}\label{lem4.1}
If $F(\psi )=0$ then $\mathrm{dim}F(C(\psi ))=1$ and if $F(\psi )\not=0$ then $F(C(\psi ))=0$.
\end{lemma}
\begin{proof}
We assume $F(\psi )=0$. By Proposition \ref{pro3.3}, $F(E_{(0,\mu )})=0$ and so $F(\iota )=0$. Thus, $F(\pi )$ is the isomorphism because of the exactness of the sequence, and dim$F(E_{(1,\nu )})=1$ by Proposition \ref{pro3.3}. Hence there is a morphism which is not zero in $F(C(\psi ))$, too, so it is clear that dim$F(C(\psi ))=1$. If $F(\psi )\not=0$, then $F(\psi )$ is the isomorphism because dim$F(E_{(1,\nu )})=\mathrm{dim}F(TE_{(0,\mu )})=1$, so $F(\pi )=0$. Thus, it can be seen that $F(C(\psi ))=0$.
\end{proof}
Similarly as above, we obtain the following lemma.
\begin{lemma}\label{lem4.2}
If $G(T^{-1}\psi )=0$ then $\mathrm{dim}G(C(\psi ))=1$ and if $G(T^{-1}\psi )\not=0$ then $G(C(\psi ))=0$. 
\end{lemma}
Note that $C(\psi )$ does not depend on the choice of a non-trivial $\psi $. We consider the local expression of the morphism $\psi =\tilde{\psi }(x,y)d\bar{z} $ as follows. For $\tilde{\psi }(x,y)$, it can be Fourier-expanded as 
\begin{equation*}
\tilde{\psi } (x,y)=\displaystyle{\sum_{H\in \mathbb{Z}}}\psi _H(x)e^{\mathbf{i}Hy},
\end{equation*}
and we see that $\psi _H(x)$ satisfies 
\begin{equation}
\psi _H(x+2\pi )=\psi _{H+1}(x) \label{psi}
\end{equation}
by the conditions of transition functions of $E_{(0,\mu )}$ and $E_{(1,\nu )}$. So we need to define $\psi _H(x)$. Here, as a non-trivial morphism in $\mathrm{Ext}^1(E_{(1,\nu )},E_{(0,\mu )})$, we take a bump function $\psi _H(x)$ as described in Figure 1, where $\varepsilon \in \mathbb{R}$ satisfies $0<\varepsilon <\pi $. The bump function $\psi _H : \mathbb{R}\rightarrow \mathbb{R}$ monotonically increases from $-2\pi H+p-s-\varepsilon $ to $-2\pi H+p-s$, and monotonically decreases from $-2\pi H+p-s$ to $-2\pi H+p-s+\varepsilon $, and otherwise it takes value zero.
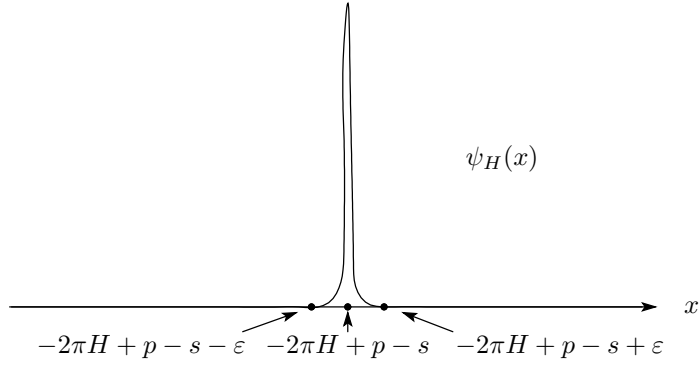
\begin{figure}
\begin{center}
{\unitlength 0.1in%
\begin{picture}(34.8000,26.1000)(0.0000,-26.1000)%
%
\special{pn 8}%
\special{pa 0 2200}%
\special{pa 3380 2200}%
\special{fp}%
\special{sh 1}%
\special{pa 3380 2200}%
\special{pa 3313 2180}%
\special{pa 3327 2200}%
\special{pa 3313 2220}%
\special{pa 3380 2200}%
\special{fp}%
\put(17.7000,-24.1000){\makebox(0,0){$-2\pi H+p-s$}}%
\put(35.7000,-22.0000){\makebox(0,0){$x$}}%
%
\special{pn 8}%
\special{pa 35502652 1634780}%
\special{pa 35502652 1634780}%
\special{fp}%
%
\special{pn 8}%
\special{pa 1634876 1977853670}%
\special{pa 1634876 1977853670}%
\special{fp}%
\put(6.9000,-24.1000){\makebox(0,0){$-2\pi H+p-s-\varepsilon$}}%
\put(28.8000,-24.1000){\makebox(0,0){$-2\pi H+p-s+\varepsilon$}}%
\put(25.8000,-14.2000){\makebox(0,0){$\psi_H(x)$}}%
%
\special{pn 8}%
\special{pa 10 2200}%
\special{pa 1194 2200}%
\special{pa 1226 2201}%
\special{pa 1386 2201}%
\special{pa 1450 2199}%
\special{pa 1482 2199}%
\special{pa 1514 2200}%
\special{pa 1546 2203}%
\special{pa 1578 2204}%
\special{pa 1610 2202}%
\special{pa 1641 2193}%
\special{pa 1670 2178}%
\special{pa 1693 2156}%
\special{pa 1711 2129}%
\special{pa 1724 2100}%
\special{pa 1734 2070}%
\special{pa 1741 2039}%
\special{pa 1745 2007}%
\special{pa 1748 1975}%
\special{pa 1749 1942}%
\special{pa 1750 1910}%
\special{pa 1750 1877}%
\special{pa 1751 1844}%
\special{pa 1751 1811}%
\special{pa 1752 1777}%
\special{pa 1752 1678}%
\special{pa 1753 1645}%
\special{pa 1753 1484}%
\special{pa 1752 1453}%
\special{pa 1752 1391}%
\special{pa 1751 1361}%
\special{pa 1751 1331}%
\special{pa 1750 1302}%
\special{pa 1749 1274}%
\special{pa 1748 1245}%
\special{pa 1745 1161}%
\special{pa 1745 1132}%
\special{pa 1744 1102}%
\special{pa 1744 975}%
\special{pa 1745 939}%
\special{pa 1746 902}%
\special{pa 1748 863}%
\special{pa 1750 822}%
\special{pa 1756 736}%
\special{pa 1759 695}%
\special{pa 1762 660}%
\special{pa 1765 632}%
\special{pa 1768 615}%
\special{pa 1771 610}%
\special{pa 1773 618}%
\special{pa 1775 638}%
\special{pa 1776 667}%
\special{pa 1777 704}%
\special{pa 1778 746}%
\special{pa 1779 792}%
\special{pa 1780 839}%
\special{pa 1780 885}%
\special{pa 1781 930}%
\special{pa 1782 973}%
\special{pa 1782 1015}%
\special{pa 1783 1056}%
\special{pa 1784 1095}%
\special{pa 1784 1133}%
\special{pa 1785 1170}%
\special{pa 1786 1205}%
\special{pa 1786 1240}%
\special{pa 1787 1274}%
\special{pa 1788 1307}%
\special{pa 1788 1339}%
\special{pa 1790 1401}%
\special{pa 1790 1431}%
\special{pa 1791 1460}%
\special{pa 1791 1489}%
\special{pa 1793 1545}%
\special{pa 1793 1573}%
\special{pa 1794 1600}%
\special{pa 1794 1628}%
\special{pa 1795 1655}%
\special{pa 1795 1682}%
\special{pa 1796 1709}%
\special{pa 1796 1737}%
\special{pa 1797 1764}%
\special{pa 1797 1792}%
\special{pa 1798 1820}%
\special{pa 1798 1848}%
\special{pa 1799 1877}%
\special{pa 1799 1936}%
\special{pa 1800 1966}%
\special{pa 1800 1997}%
\special{pa 1801 2029}%
\special{pa 1804 2061}%
\special{pa 1811 2093}%
\special{pa 1823 2124}%
\special{pa 1840 2152}%
\special{pa 1863 2175}%
\special{pa 1889 2190}%
\special{pa 1920 2197}%
\special{pa 1953 2200}%
\special{pa 2020 2200}%
\special{pa 2053 2199}%
\special{pa 2149 2199}%
\special{pa 2180 2200}%
\special{pa 3310 2200}%
\special{fp}%
%
\special{pn 4}%
\special{sh 1}%
\special{ar 1770 2200 16 16 0 6.2831853}%
\special{sh 1}%
\special{ar 1770 2200 16 16 0 6.2831853}%
%
\special{pn 4}%
\special{sh 1}%
\special{ar 1580 2200 16 16 0 6.2831853}%
\special{sh 1}%
\special{ar 1580 2200 16 16 0 6.2831853}%
%
\special{pn 4}%
\special{sh 1}%
\special{ar 1960 2200 16 16 0 6.2831853}%
\special{sh 1}%
\special{ar 1960 2200 16 16 0 6.2831853}%
%
\special{pn 8}%
\special{pa 1220 2330}%
\special{pa 1500 2230}%
\special{fp}%
\special{sh 1}%
\special{pa 1500 2230}%
\special{pa 1430 2234}%
\special{pa 1450 2248}%
\special{pa 1444 2271}%
\special{pa 1500 2230}%
\special{fp}%
%
\special{pn 8}%
\special{pa 2320 2330}%
\special{pa 2060 2230}%
\special{fp}%
\special{sh 1}%
\special{pa 2060 2230}%
\special{pa 2115 2273}%
\special{pa 2110 2249}%
\special{pa 2129 2235}%
\special{pa 2060 2230}%
\special{fp}%
%
\special{pn 8}%
\special{pa 1770 2330}%
\special{pa 1770 2230}%
\special{fp}%
\special{sh 1}%
\special{pa 1770 2230}%
\special{pa 1750 2297}%
\special{pa 1770 2283}%
\special{pa 1790 2297}%
\special{pa 1770 2230}%
\special{fp}%
\end{picture}}%
\caption{An example of the bump function $\psi _H(x)$}
\end{center}
\end{figure}
Note that we can extend $\tilde{\psi} (x,y)$ defined locally to a function defined on $\mathbb{R}^2$ by using the relation (\ref{psi}). In this sense, we often treat $\tilde{\psi } (x,y)$ as a function on $\mathbb{R}^2$. For this $\psi $, in the symplectic geometry side, the values $-2\pi H+p-s$ ($H\in \mathbb{Z}$) correspond to the $x$ coordinates of the intersection points of the Lagrangian submanifolds $\mathcal{L}_{(0,p)}$, $\mathcal{L}_{(1,s)}$ and their copies in the covering space of $T^2$.

We examine the condition for $\eta $ to satisfy dim$F(C(\psi ))=\mathrm{dim}G(C(\psi ))=1$, i.e., the condition to exist non-trivial holomorphic maps $\tilde{\phi } : E_{(\frac{1}{2},\frac{\eta }{2})}\rightarrow C(\psi )$ and $\phi : C(\psi )\rightarrow E_{(\frac{1}{2},\frac{\eta }{2})}$. For $\tilde{\phi } : E_{(\frac{1}{2},\frac{\eta }{2})}\rightarrow C(\psi )$, we obtain the following theorem.
\begin{theo}\label{the4.3}
For $E_{(\frac{1}{2},\frac{\eta }{2})}$, $\mathrm{dim}F(C(\psi ))=1$ holds if and only if $\eta \equiv \mu +\nu +\pi +\pi \tau \ (\mathrm{mod}\ 2\pi (\mathbb{Z}\oplus \tau \mathbb{Z}))$.
\end{theo}
\begin{proof}
We recall the definition of the mapping cone of $\psi $. It is defined by
\begin{align*}
&C(\psi ):=E_{(0,\mu )}\oplus E_{(1,\nu )},\\
&\tilde{d}:=\left(\begin{array}{ccc}d_0&\psi \\0&d_1\end{array}\right)=\left(\begin{array}{ccc}2\bar{\partial }-\frac{\mathbf{i}}{\pi (\bar{\tau }-\tau )}\mu d\bar{z}&\tilde{\psi } d\bar{z}\\0&2\bar{\partial }-\frac{\mathbf{i}}{\pi (\bar{\tau }-\tau )}(x+\nu )d\bar{z}\end{array}\right).
\end{align*}
Here, the transition functions of $C(\psi )$ on $O_{1j}\cap O_{2j}$, $O_{2j}\cap O_{3j}$, $O_{i1}\cap O_{i2}$, $O_{i2}\cap O_{i3}$ and $O_{i3}\cap O_{i1}$ are trivial, but are non-trivial on $O_{3j}\cap O_{1j}$ ($i,j=1,2,3$). They are expressed as 
\begin{equation}
\left.\left(\begin{array}{ccc}\tilde{s}_0\\\tilde{s}_1\end{array}\right)_{(3,j)} \right|_{O_{3j}\cap O_{1j}}=\left(\begin{array}{ccc}1&0\\0&e^{\mathbf{i}y}\end{array}\right) \left.\left(\begin{array}{ccc}\tilde{s}_0\\\tilde{s}_1\end{array}\right)_{(1,j)} \right|_{O_{3j}\cap O_{1j}}, \label{transition}
\end{equation}
where $(\tilde{s}_0,\tilde{s}_1)^t_{(i,j)} \in \Gamma (C(\psi ))$. Note that $\Gamma (C(\psi ))$ denotes the set of sections of $C(\psi )$. On the other hand, the holomorphic structure of $E_{(\frac{1}{2},\frac{\eta }{2})}\ (\eta =u+v\tau ,u,v\in \mathbb{R})$ is defined by 
\begin{equation*}
d:=\left(\begin{array}{ccc}d_{\frac{1}{2}}&0\\0&d_{\frac{1}{2}}\end{array}\right)=\left(\begin{array}{ccc}2\bar{\partial }-\frac{\mathbf{i}}{\pi (\bar{\tau } -\tau )}(\frac{1}{2}x+\frac{\eta}{2} )d\bar{z}&0\\0&2\bar{\partial }-\frac{\mathbf{i}}{\pi (\bar{\tau } -\tau )}(\frac{1}{2}x+\frac{\eta }{2})d\bar{z}\end{array}\right).
\end{equation*}
The transition functions of $E_{(\frac{1}{2},\frac{\eta }{2})}$ are trivial on $O_{1j}\cap O_{2j}$, $O_{2j}\cap O_{3j}$, $O_{i1}\cap O_{i2}$ and $O_{i2}\cap O_{i3}$, but are non-trivial on $O_{3j}\cap O_{1j}$ and $O_{i3}\cap O_{i1}$. They are expressed as
\begin{align}
&\left.\left(\begin{array}{ccc}s_0\\s_1\end{array}\right)_{(3,j)}\right|_{O_{3j}\cap O_{1j}}=\left(\begin{array}{ccc}0&e^{\frac{\mathbf{i}}{2}y}\\e^{\frac{\mathbf{i}}{2}y}&0\end{array}\right) \left.\left(\begin{array}{ccc}s_0\\s_1\end{array}\right)_{(1,j)}\right|_{O_{3j}\cap O_{1j}}\label{s1},\\
&\left.\left(\begin{array}{ccc}s_0\\s_1\end{array}\right)_{(i,3)}\right|_{O_{i3}\cap O_{i1}}=\left(\begin{array}{ccc}1&0\\0&-1\end{array}\right)\left.\left(\begin{array}{ccc}s_0\\s_1\end{array}\right)_{(i,1)}\right|_{O_{i3}\cap O_{i1}},\label{s2}
\end{align}
where $(s_0,s_1)^t_{(i,j)} \in \Gamma (E_{(\frac{1}{2},\frac{\eta }{2})})$.
By definition, $\mathrm{rank}C(\psi )=\mathrm{rank}E_{(\frac{1}{2},\frac{\eta }{2})}=2$, so $\tilde{\phi }$ is expressed locally as
\begin{equation*}
\tilde{\phi }=\left(\begin{array}{ccc}\tilde{\phi }_{11}(x,y)&\tilde{\phi }_{12}(x,y)\\\tilde{\phi }_{21}(x,y)&\tilde{\phi }_{22}(x,y)\end{array}\right).
\end{equation*}
For the transition function of $C(\psi )$ on $O_{i3}\cap O_{i1}$, we see that
\begin{equation}
\left.\left(\begin{array}{ccc}\tilde{s}_0\\\tilde{s}_1\end{array}\right)_{(i,3)}\right|_{O_{i3}\cap O_{i1}} =\left.\left(\begin{array}{ccc}\tilde{s}_0\\\tilde{s}_1\end{array}\right)_{(i,1)}\right|_{O_{i3}\cap O_{i1}}, \label{trans}
\end{equation}
so by using the relations (\ref{s2}), (\ref{trans}), it can be seen that $\tilde{\phi }_{11}$, $\tilde{\phi }_{12}$, $\tilde{\phi }_{21}$ and $\tilde{\phi }_{22}$ satisfy 
 \begin{gather*}
\tilde{\phi }_{11}(x,y+2\pi )=\tilde{\phi }_{11}(x,y),\ -\tilde{\phi }_{12}(x,y+2\pi )=\tilde{\phi }_{12}(x,y),\\
\tilde{\phi }_{21}(x,y+2\pi )=\tilde{\phi }_{21}(x,y),\ -\tilde{\phi }_{22}(x,y+2\pi )=\tilde{\phi }_{22}(x,y).
\end{gather*}
Hence they can be Fourier-expanded as 
\begin{gather*}
\tilde{\phi  }_{11}(x,y)=\displaystyle{\sum_{I\in \mathbb{Z}}}\tilde{\phi }_{11,I}(x)e^{\mathbf{i}Iy},\ \tilde{\phi }_{12}(x,y)=e^{\frac{\mathbf{i}}{2}y}\displaystyle{\sum_{J\in \mathbb{Z}}}\tilde{\phi }_{12,J}(x)e^{\mathbf{i}Jy},\\
\tilde{\phi }_{21}(x,y)=\displaystyle{\sum_{K\in \mathbb{Z}}}\tilde{\phi }_{21}(x)e^{\mathbf{i}Ky},\ \tilde{\phi }_{22}(x,y)=e^{\frac{\mathbf{i}}{2}y}\displaystyle{\sum_{L\in \mathbb{Z}}}\tilde{\phi }_{22,L}(x)e^{\mathbf{i}Ly}.
\end{gather*}
Furthermore, by using the relations (\ref{transition}), (\ref{s1}), we obtain 
\begin{align*}
&\tilde{\phi }_{11,I}(x+2\pi )=\tilde{\phi }_{12,I}(x),\ \tilde{\phi }_{12,J}(x+2\pi )=\tilde{\phi }_{11,J+1}(x),\\
&\tilde{\phi }_{21,K}(x+2\pi )=\tilde{\phi }_{22,K-1}(x),\ \tilde{\phi }_{22,L}(x+2\pi )=\tilde{\phi }_{21,L}(x). 
\end{align*}
Now $\tilde{\phi }$ belongs to $H^0(E_{(\frac{1}{2},\frac{\eta }{2})},C(\psi ))$, namely, $\tilde{\phi }$ satisfies the differential equation $\tilde{d}\tilde{\phi }-\tilde{\phi }d=0$. We see 
\begin{align*}
&\tilde{d}\tilde{\phi }-\tilde{\phi }d\\
&=\left(\begin{array}{ccc}d_0&\psi \\0&d_1\end{array}\right)\left(\begin{array}{ccc}\tilde{\phi }_{11}&\tilde{\phi }_{12}\\\tilde{\phi }_{21}&\tilde{\phi }_{22}\end{array}\right)-\left(\begin{array}{ccc}\tilde{\phi }_{11}&\tilde{\phi }_{12}\\\tilde{\phi }_{21}&\tilde{\phi }_{22}\end{array}\right)\left(\begin{array}{ccc}d_{\frac{1}{2}}&0\\0&d_{\frac{1}{2}}\end{array}\right)\\
&=\left(\begin{array}{ccc}d_0\tilde{\phi }_{11}-\tilde{\phi }_{11}d_{\frac{1}{2}}+\tilde{\psi } \tilde{\phi }_{21}d\bar{z}&d_0\tilde{\phi }_{12}-\tilde{\phi }_{12}d_{\frac{1}{2}}+\tilde{\psi } \tilde{\phi }_{22}d\bar{z}\\d_1\tilde{\phi }_{21}-\tilde{\phi }_{21}d_{\frac{1}{2}}&d_1\tilde{\phi }_{22}-\tilde{\phi }_{22}d_{\frac{1}{2}}\end{array}\right).
\end{align*}
Each component of this matrix is expressed locally as follows.
\begin{flalign*}
&d_0\tilde{\phi }_{11}-\tilde{\phi }_{11}d_{\frac{1}{2}}+\tilde{\psi } \tilde{\phi }_{21}d\bar{z}&\\
&=\displaystyle{-\frac{2\tau }{\bar{\tau }-\tau }}\displaystyle{\sum_{I\in \mathbb{Z}}}\Bigl\{\frac{d}{dx}\tilde{\phi }_{11,I}(x)-\frac{\mathbf{i}}{\tau }\left(\frac{x}{4\pi }-\frac{p}{2\pi }+\frac{u}{4\pi }-\frac{q}{2\pi }\tau +\frac{v}{4\pi }\tau +I\right)\tilde{\phi }_{11,I}(x)&\\
&\ \ \ \ -\frac{\bar{\tau }-\tau }{2\tau }\displaystyle{\sum_{H\in \mathbb{Z}}}\psi _H(x)\tilde{\phi }_{21,I-H}(x)\Bigr\}e^{\mathbf{i}Iy}d\bar{z},&
\end{flalign*}
\begin{flalign*}
&d_0\tilde{\phi }_{12}-\tilde{\phi }_{12}d_{\frac{1}{2}}+\tilde{\psi } \tilde{\phi }_{22}d\bar{z}&\\
&=-\frac{2\tau }{\bar{\tau }-\tau }e^{\frac{i}{2}y}\displaystyle{\sum_{J\in \mathbb{Z}}}\Bigl\{\frac{d}{dx}\tilde{\phi }_{12,J}(x)&\\
&\ \ \ -\frac{\mathbf{i}}{\tau }\left(\frac{x}{4\pi }-\frac{p}{2\pi }+\frac{u}{4\pi }+\frac{1}{2}-\frac{q}{2\pi }\tau +\frac{v}{4\pi }\tau +J\right)\tilde{\phi }_{12,J}(x)&\\
&\ \ \ -\frac{\bar{\tau }-\tau }{2\tau }\displaystyle{\sum_{H\in \mathbb{Z}}}\psi _H(x)\tilde{\phi }_{22,J-H}(x)\Bigr\}e^{\mathbf{i}Jy}d\bar{z},&\\
&d_1\tilde{\phi }_{21}-\tilde{\phi }_{21}d_{\frac{1}{2}}&\\
&=-\frac{2\tau }{\bar{\tau }-\tau }\displaystyle{\sum_{K\in \mathbb{Z}}}\Bigl\{\frac{d}{dx}\tilde{\phi }_{21,K}(x)&\\
&\ \ \ +\frac{\mathbf{i}}{\tau }\left(\frac{x}{4\pi }+\frac{s}{2\pi }-\frac{u}{4\pi }+\frac{t}{2\pi }\tau -\frac{v}{4\pi }\tau  -K\right)\tilde{\phi }_{21,K}(x)\Bigr\}e^{\mathbf{i}Ky}d\bar{z},&\\
&d_1\tilde{\phi }_{22}-\tilde{\phi }_{22}d_{\frac{1}{2}}&\\
&=-\frac{2\tau }{\bar{\tau }-\tau }e^{\frac{\mathbf{i}}{2}y}\displaystyle{\sum_{L\in \mathbb{Z}}}\Bigl\{\frac{d}{dx}\tilde{\phi }_{22,L}(x)&\\
&\ \ \ +\frac{\mathbf{i}}{\tau }\left(\frac{x}{4\pi }+\frac{s}{2\pi }-\frac{u}{4\pi }-\frac{1}{2}+\frac{t}{2\pi }\tau -\frac{v}{4\pi }\tau -L\right)\tilde{\phi }_{22,L}(x)\Bigr\}e^{\mathbf{i}Ly}d\bar{z}.&
\end{flalign*}
During these calculations, we put $K+H=I$ and $L+H=J$. The solutions of the differential equations for $\tilde{\phi }_{21,K}(x)$ and $\tilde{\phi }_{22,L}(x)$ which satisfy $\tilde{\phi }_{21,K}(x+2\pi )=\tilde{\phi }_{22,K-1}(x)$ and $\tilde{\phi }_{22,L}(x+2\pi )=\tilde{\phi }_{21,L}(x)$ are given by
\begin{align*}
&\tilde{\phi }_{21,K}(x)=\tilde{C}_{21,K}e^{-\frac{\mathbf{i}}{\tau }(\frac{x^2}{8\pi }+(\frac{s}{2\pi }-\frac{u}{4\pi }+\frac{t}{2\pi }\tau -\frac{v}{4\pi }\tau  -K)x)},\\
&\tilde{\phi }_{22,L}(x)=\tilde{C}_{22,L}e^{-\frac{\mathbf{i}}{\tau }(\frac{x^2}{8\pi }+(\frac{s}{2\pi }-\frac{u}{4\pi }-\frac{1}{2}+\frac{t}{2\pi }\tau -\frac{v}{4\pi }\tau  -L)x)}.
\end{align*}
Here $\tilde{C}_{21,K}$ and $\tilde{C}_{22,L}$ are arbitrary constants and they satisfy 
\begin{align*}
&\tilde{C}_{21,K}e^{\frac{\mathbf{i}}{\tau }(2\pi K-s+\frac{u}{2}-\frac{\pi }{2}-t\tau +\frac{v}{2}\tau )}=\tilde{C}_{22,K-1},\\
&\tilde{C}_{22,L}e^{\frac{\mathbf{i}}{\tau }(2\pi L-s+\frac{u}{2}+\frac{\pi }{2}-t\tau +\frac{v}{2}\tau )}=\tilde{C}_{21,L}. 
\end{align*}
Hence, $\tilde{C}_{21,K}$ and $\tilde{C}_{22,L}$ are given by
 \begin{align*}
&\tilde{C}_{21,K}=e^{-\frac{\mathbf{i}}{\tau }(2\pi K^2-(2s-u+2t\tau -v\tau )K)},\\
&\tilde{C}_{22,L}=e^{-\frac{\mathbf{i}}{\tau }(2\pi L^2-(2s-u-2\pi +2t\tau -v\tau  )L-s+\frac{u}{2}+\frac{\pi }{2}-t\tau +\frac{v}{2}\tau )}.
\end{align*}
Thus, $\tilde{\phi }_{21,K}(x)$ and $\tilde{\phi }_{22,L}(x)$ are expressed locally as
\begin{align*}
&\tilde{\phi }_{21,K}(x)=e^{-\frac{\mathbf{i}}{4\pi }(2t-v)x+\frac{\mathbf{i}}{8\pi \tau }(u-2s)^2+\mathbf{i}(2t-v)K-\frac{\mathbf{i}}{8\pi \tau }(x-4\pi K+2s-u)^2},\\
&\tilde{\phi }_{22,L}(x)=e^{-\frac{\mathbf{i}}{4\pi }(2t-v)x+\frac{\mathbf{i}}{8\pi \tau }(u-2s)^2+\frac{\mathbf{i}}{2}(2t-v)+\mathbf{i}(2t-v)L-\frac{\mathbf{i}}{8\pi \tau }(x-4\pi L+2s-u-2\pi )^2},
\end{align*}
so $\tilde{\phi }_{21}(x,y)$ and $\tilde{\phi }_{22}(x,y)$ are expressed locally as follows.
 \begin{align*}
&\tilde{\phi }_{21}(x,y)=e^{-\frac{\mathbf{i}}{4\pi }(2t-v)x+\frac{\mathbf{i}}{8\pi \tau }(u-2s)^2}\displaystyle{\sum_{K\in \mathbb{Z}}}e^{\mathbf{i}(2t-v)K-\frac{\mathbf{i}}{8\pi \tau }(x-4\pi K+2s-u)^2}e^{\mathbf{i}Ky},\\
&\tilde{\phi }_{22}(x,y)=e^{-\frac{\mathbf{i}}{4\pi }(2t-v)x+\frac{\mathbf{i}}{2}y+\frac{\mathbf{i}}{8\pi \tau }(u-2s)^2+\frac{\mathbf{i}}{2}(2t-v)}\\
&\hspace{17mm}\times \displaystyle{\sum_{L\in \mathbb{Z}}}e^{\mathbf{i}(2t-v)L-\frac{\mathbf{i}}{8\pi \tau }(x-4\pi L+2s-u-2\pi )^2}e^{\mathbf{i}Ly}.
\end{align*}
The solutions of the differential equations for $\tilde{\phi }_{11,I}(x)$ and $\tilde{\phi }_{12,J}(x)$ are given by 
\begin{align*}
&\tilde{\phi }_{11,I}(x)=\Bigl\{\frac{\bar{\tau }-\tau }{2\tau }\displaystyle{\sum_{H\in \mathbb{Z}}}\tilde{C}_{21,I-H}\displaystyle{\int_{-\infty }^xe^{-\frac{\mathbf{i}}{\tau }(\frac{x^2}{4\pi }+(-\frac{p}{2\pi }+\frac{s}{2\pi }-\frac{q}{2\pi }\tau +\frac{t}{2\pi }\tau +H)x)}\psi _H(x)dx}\\
&\hspace{17mm}+\tilde{C}_{11,I}\Bigr\}e^{\frac{\mathbf{i}}{\tau }(\frac{x^2}{8\pi }+(-\frac{p}{2\pi }+\frac{u}{4\pi }-\frac{q}{2\pi }\tau +\frac{v}{4\pi }\tau +I)x)},\\
&\tilde{\phi }_{12,J}(x)=\Bigl\{\frac{\bar{\tau }-\tau }{2\tau }\displaystyle{\sum_{H\in \mathbb{Z}}}\tilde{C}_{22,J-H}\displaystyle{\int_{-\infty }^xe^{-\frac{\mathbf{i}}{\tau }(\frac{x^2}{4\pi }+(-\frac{p}{2\pi }+\frac{s}{2\pi }-\frac{q}{2\pi }\tau +\frac{t}{2\pi }\tau +H)x)}\psi _H(x)dx}\\
&\hspace{17mm}+\tilde{C}_{12,J}\Bigr\}e^{\frac{\mathbf{i}}{\tau }(\frac{x^2}{8\pi }+(-\frac{p}{2\pi }+\frac{u}{4\pi }+\frac{1}{2}-\frac{q}{2\pi }\tau +\frac{v}{4\pi }\tau +J)x)},
\end{align*}
where $\tilde{C}_{11,I}$ and $\tilde{C}_{12,J}$ are arbitrary constants. Here, we introduce a function $\theta _{\varepsilon }(x-(-2\pi H+p-s))$ by
 \begin{equation*}
\lambda _{\tau ,H}\theta _{\varepsilon }(x-(-2\pi H+p-s)):=\displaystyle{\int_{-\infty }^xe^{-\frac{\mathbf{i}}{\tau }(\frac{x^2}{4\pi }+(-\frac{p}{2\pi }+\frac{s}{2\pi }-\frac{q}{2\pi }\tau +\frac{t}{2\pi }\tau +H)x)}\psi _H(x)dx},
\end{equation*}
where $\varepsilon \in \mathbb{R}$ satisfies $0<\varepsilon <\pi $ and $\lambda _{\tau ,H}\in \mathbb{C}$. We assume that for any $x\geq -2\pi H+p-s+\varepsilon $, $\theta _{\varepsilon }(x-(-2\pi H+p-s))=1$. Then from the periodicity constraints $\tilde{\phi }_{11,I}(x+2\pi )=\tilde{\phi }_{12,I}(x)$ and $\tilde{\phi }_{12,J}(x+2\pi )=\tilde{\phi }_{11,J+1}(x)$, we obtain the following conditions.
\begin{align*}
&\lambda _{\tau ,H}=e^{\frac{\mathbf{i}}{\tau }(2\pi H-p+s-\pi -q\tau +t\tau )}\lambda _{\tau ,H-1},\\
&\tilde{C}_{11,I}e^{\frac{\mathbf{i}}{\tau }(2\pi I-p+2u+\frac{\pi }{2}-q\tau +2v\tau )}=\tilde{C}_{12,I},\\
&\tilde{C}_{12,J}e^{\frac{\mathbf{i}}{\tau }(2\pi J-p+2u+\frac{3}{2}\pi -q\tau +2v\tau )}=\tilde{C}_{11,J+1}.
\end{align*}
Hence, it can be seen that $\lambda _{\tau ,H}=e^{\frac{\mathbf{i}}{\tau }(\pi H^2+(-p+s-q\tau +t\tau )H+\lambda )}$, where $\lambda \in \mathbb{C}$. Thus, we obtain the following formula.
\begin{align*}
&\displaystyle{\int_{-\infty }^xe^{-\frac{\mathbf{i}}{\tau }(\frac{x^2}{4\pi }+(-\frac{p}{2\pi }+\frac{s}{2\pi }-\frac{q}{2\pi }\tau +\frac{t}{2\pi }\tau +H)x)}\psi _H(x)dx}\\
&=e^{\frac{\mathbf{i}}{\tau }(\pi H^2+(-p+s-q\tau +t\tau )H+\lambda )}\theta _{\varepsilon }(x-(-2\pi H+p-s)).
\end{align*}
We put $\lambda =\frac{q}{2}\tau +\frac{t}{2}\tau -\frac{v}{2}\tau -\frac{\pi }{4}$. The solutions $\tilde{\phi }_{11,I}(x)$ and $\tilde{\phi }_{12,J}(x)$ must satisfy 
\begin{equation}
\displaystyle{\lim_{x\rightarrow \pm \infty }}\tilde{\phi }_{11,I}(x)=0, \displaystyle{\lim_{x\rightarrow \pm \infty }}\tilde{\phi }_{12,J}(x)=0, \label{limit}
\end{equation}
because, for any $(\alpha ,\beta )\in \mathbb{R}^2$, the values of $\tilde{\phi }_{11}(\alpha ,\beta )$ and $\tilde{\phi }_{12}(\alpha ,\beta )$ must not diverge under the conditions $\tilde{\phi }_{11,I}(x+2\pi )=\tilde{\phi }_{12,I}(x)$ and $\tilde{\phi }_{12,J}(x+2\pi )=\tilde{\phi }_{11,J+1}(x)$. We examine when $u$ and $v$ satisfy the convergence conditions (\ref{limit}). First, for $\tilde{\phi }_{11,I}(x)$, one has
\begin{flalign*}
&\tilde{\phi }_{11,I}(x)=\frac{\bar{\tau }-\tau }{2\tau }e^{\frac{\mathbf{i}}{4\pi }(v-2q)x-\mathbf{i}(v-2t)I-\frac{\mathbf{i}}{\tau }(\frac{p^2}{4\pi }-\frac{s^2}{4\pi }-\frac{u^2}{8\pi }-\frac{ps}{2\pi }+\frac{us}{2\pi }+\frac{\pi }{4})}&\\
&\ \ \ \ \ \ \ \ \ \ \ \ \ \ \times \Bigl\{\displaystyle{\sum_{H\in \mathbb{Z}}}e^{\frac{\mathbf{i}}{2}(v-q-t)(2H-1)-\frac{\pi \mathbf{i}}{\tau }(H-2I+\frac{p}{2\pi }+\frac{s}{2\pi }-\frac{u}{2\pi })^2}\theta _{\varepsilon }(x-(-2\pi H+p-s))&\\
&\ \ \ \ \ \ \ \ \ \ \ \ \ \ +\frac{2\tau }{\bar{\tau }-\tau }e^{-\frac{\mathbf{i}}{\tau }\{2\pi I^2+(-2p+u+2t\tau -v\tau )I+\frac{p^2}{4\pi }+\frac{s^2}{4\pi }+\frac{u^2}{4\pi }+\frac{ps}{4\pi }-\frac{pu}{2\pi }-\frac{us}{2\pi }-\frac{\pi }{4}\}}\tilde{C}_{11,I}\Bigr\}&\\
&\ \ \ \ \ \ \ \ \ \ \ \ \ \ \times e^{\frac{\mathbf{i}}{8\pi \tau }(x+4\pi I-2p+u)^2}.&
\end{flalign*}
Here, the function 
\begin{equation*}
\displaystyle{\sum_{H\in \mathbb{Z}}}e^{\frac{\mathbf{i}}{2}(v-q-t)(2H-1)-\frac{\pi \mathbf{i}}{\tau }(H-2I+\frac{p}{2\pi }+\frac{s}{2\pi }-\frac{u}{2\pi })^2}\theta _{\varepsilon }(x-(-2\pi H+p-s))
\end{equation*}
converges to $0$ when $x\rightarrow -\infty $. Thus, it can be seen that $\tilde{C}_{11,I}=0$ for any $I\in \mathbb{Z}$, so we check the conditions for $u$ and $v$ to satisfy 
\begin{equation*}
\displaystyle{\lim_{x\rightarrow \infty }}\displaystyle{\sum_{H\in \mathbb{Z}}}e^{\frac{\mathbf{i}}{2}(v-q-t)(2H-1)-\frac{\pi \mathbf{i}}{\tau }(H-2I+\frac{p}{2\pi }+\frac{s}{2\pi }-\frac{u}{2\pi })^2}\theta _{\varepsilon }(x-(-2\pi H+p-s))=0. 
\end{equation*}
This limiting value is calculated as 
\begin{align*}
&\displaystyle{\lim_{x\rightarrow \infty }}\displaystyle{\sum_{H\in \mathbb{Z}}}e^{\frac{\mathbf{i}}{2}(v-q-t)(2H-1)-\frac{\pi \mathbf{i}}{\tau }(H-2I+\frac{p}{2\pi }+\frac{s}{2\pi }-\frac{u}{2\pi })^2}\theta _{\varepsilon }(x-(-2\pi H+p-s))\\
&=e^{\mathbf{i}(v-q-t)(2I+\frac{1}{2})}\displaystyle{\sum_{n\in \mathbb{Z}}}e^{2\pi \mathbf{i}n(-\frac{q}{2\pi }-\frac{t}{2\pi }+\frac{v}{2\pi })-\frac{\pi \mathbf{i}}{\tau }(n+1+\frac{p}{2\pi }+\frac{s}{2\pi }-\frac{u}{2\pi })^2}\\
&=e^{\mathbf{i}(v-q-t)(2I-\frac{p}{2\pi }-\frac{s}{2\pi }+\frac{u}{2\pi }-\frac{1}{2})}\displaystyle{\sum_{n\in \mathbb{Z}}}e^{-\frac{\pi \mathbf{i}}{\tau }(n+\frac{p}{2\pi }+\frac{s}{2\pi }-\frac{u}{2\pi }+1)^2}\\
&\ \ \ \times e^{2\pi \mathbf{i}(n+\frac{p}{2\pi }+\frac{s}{2\pi }-\frac{u}{2\pi }+1)(-\frac{q}{2\pi }-\frac{t}{2\pi }+\frac{v}{2\pi }-\frac{1}{2}+\frac{1}{2})}\\
&=e^{\mathbf{i}(v-q-t)(2I-\frac{p}{2\pi }-\frac{s}{2\pi }+\frac{u}{2\pi }-\frac{1}{2})}\vartheta _{\frac{1}{2}+(\frac{p}{2\pi }+\frac{s}{2\pi }-\frac{u}{2\pi }+\frac{1}{2}),\frac{1}{2}}\left(-\frac{q}{2\pi }-\frac{t}{2\pi }+\frac{v}{2\pi }-\frac{1}{2},-\frac{1}{\tau }\right)\\
&=e^{\mathbf{i}(v-q-t)(2I-\frac{p}{2\pi }-\frac{s}{2\pi }+\frac{u}{2\pi }-\frac{1}{2})-\frac{\pi \mathbf{i}}{\tau }(\frac{p}{2\pi }+\frac{s}{2\pi }-\frac{u}{2\pi }+\frac{1}{2})^2+2\pi \mathbf{i}(\frac{p}{2\pi }+\frac{s}{2\pi }-\frac{u}{2\pi }+\frac{1}{2})(-\frac{q}{2\pi }-\frac{t}{2\pi }+\frac{v}{2\pi }-\frac{1}{2})}\\
&\ \ \ \times e^{\pi \mathbf{i}(\frac{p}{2\pi }+\frac{s}{2\pi }-\frac{u}{2\pi }+\frac{1}{2})}\\
&\ \ \ \times \vartheta _{\frac{1}{2},\frac{1}{2}}\left(-\frac{q}{2\pi }-\frac{t}{2\pi }+\frac{v}{2\pi }-\frac{1}{2}+\left(\frac{p}{2\pi }+\frac{s}{2\pi }-\frac{u}{2\pi }+\frac{1}{2}\right)\left(-\frac{1}{\tau }\right),-\frac{1}{\tau }\right).
\end{align*}
Here $\vartheta _{\frac{1}{2},\frac{1}{2}}(z,\tau )$ is the theta function. It is known that the zeros of $\vartheta _{\frac{1}{2},\frac{1}{2}}(z,\tau )$ are $z=\alpha \tau +\beta \ (\alpha ,\beta \in \mathbb{Z})$. Thus, the condition 
\begin{equation*}
\displaystyle{\lim_{x\rightarrow \infty }}\displaystyle{\sum_{H\in \mathbb{Z}}}e^{\frac{\mathbf{i}}{2}(v-q-t)(2H-1)-\frac{\pi \mathbf{i}}{\tau }(H-2I+\frac{p}{2\pi }+\frac{s}{2\pi }-\frac{u}{2\pi })^2}\theta _{\varepsilon }(x-(-2\pi H+p-s))=0
\end{equation*}
is equivalent to that $-\frac{q}{2\pi }-\frac{t}{2\pi }+\frac{v}{2\pi }-\frac{1}{2}\in \mathbb{Z}$ and $\frac{p}{2\pi }+\frac{s}{2\pi }-\frac{u}{2\pi }+\frac{1}{2}\in \mathbb{Z}$, namely, $u\equiv p+s+\pi \ (\mathrm{mod} \ 2\pi \mathbb{Z})$ and $v\equiv q+t+\pi \ (\mathrm{mod}\ 2\pi \mathbb{Z})$. To calculate similarly as above, we can check that $\displaystyle{\lim_{x\rightarrow \pm \infty }}\tilde{\phi }_{12,J}(x)=0$ holds true if and only if $\tilde{C}_{12,J}=0$ for any $J\in \mathbb{Z}$ and $u$, $v$ satisfy the relations $u\equiv p+s+\pi \ (\mathrm{mod} \ 2\pi \mathbb{Z})$, $v\equiv q+t+\pi \ (\mathrm{mod} \ 2\pi \mathbb{Z})$. Clearly, for any $I,J\in \mathbb{Z}$, $\tilde{C}_{11,I}e^{\frac{\mathbf{i}}{\tau }(2\pi I-p+2u+\frac{\pi }{2}-q\tau +2v\tau )}=\tilde{C}_{12,I}$ and $\tilde{C}_{12,J}e^{\frac{\mathbf{i}}{\tau }(2\pi J-p+2u+\frac{3}{2}\pi -q\tau +2v\tau )}=\tilde{C}_{11,J+1}$ hold in the case of $\tilde{C}_{11,I}=0$ and $\tilde{C}_{12,J}=0$. Thus, for $E_{(\frac{1}{2},\frac{\eta}{2} )}\ (\eta =u+v\tau )$, $\mathrm{dim}F(C(\psi ))=1$ holds if and only if $u\equiv p+s+\pi \ (\mathrm{mod} \ 2\pi \mathbb{Z})$, $v\equiv q+t+\pi \ (\mathrm{mod} \ 2\pi \mathbb{Z})$, namely, $\eta \equiv \mu +\nu +\pi +\pi \tau \ (\mathrm{mod}\ 2\pi (\mathbb{Z}\oplus \tau \mathbb{Z}))$.
\end{proof}
We obtain the following corollary by setting $u=p+s+\pi $ and $v=q+t+\pi $ in the proof of Theorem \ref{the4.3}.
\begin{corollary}\label{cor4.4}
Under the condition of Theorem \ref{the4.3}, we can take the base of $F(C(\psi ))$ locally expressed as
\begin{equation*}
\tilde{\phi }:=\left(\begin{array}{ccc}\tilde{\phi }_{11}&\tilde{\phi }_{12}\\\tilde{\phi }_{21}&\tilde{\phi }_{22}\end{array}\right)=\left(\begin{array}{ccc}\tilde{\phi }_{11}(x,y)&\tilde{\phi }_{12}(x,y)\\\tilde{\phi }_{21}(x,y)&\tilde{\phi }_{22}(x,y)\end{array}\right),
\end{equation*}
\begin{flalign*}
&\tilde{\phi }_{11}(x,y)&\\
&=-\frac{\bar{\tau }-\tau }{2\tau }\mathbf{i}e^{-\frac{\mathbf{i}}{4\pi }(q-t-\pi )x-\frac{\mathbf{i}}{8\pi \tau }(p-s-\pi )^2}\displaystyle{\sum_{I\in \mathbb{Z}}}\displaystyle{\sum_{H\in \mathbb{Z}}}(-1)^{I+H}e^{-\mathbf{i}(q-t)I-\frac{\pi \mathbf{i}}{\tau }(H-2I-\frac{1}{2})^2}&\\
&\ \ \ \times \theta _{\varepsilon }(x-(-2\pi H+p-s))e^{\frac{\mathbf{i}}{8\pi \tau }(x+4\pi I-p+s+\pi )^2}e^{\mathbf{i}Iy},&\\
&\tilde{\phi }_{12}(x,y)&\\
&=-\frac{\bar{\tau }-\tau }{2\tau }e^{-\frac{\mathbf{i}}{4\pi }(q-t-\pi )x+\frac{\mathbf{i}}{2}y-\frac{\mathbf{i}}{8\pi \tau }(p-s-\pi )^2-\frac{\mathbf{i}}{2}(q-t)}&\\
&\ \ \ \times \displaystyle{\sum_{J\in \mathbb{Z}}}\displaystyle{\sum_{H\in \mathbb{Z}}}(-1)^{J+H}e^{-\mathbf{i}(q-t)J-\frac{\pi \mathbf{i}}{\tau }(H-2J-\frac{3}{2})^2}&\\
&\ \ \ \times \theta _{\varepsilon }(x-(-2\pi H+p-s))e^{\frac{\mathbf{i}}{8\pi \tau }(x+4\pi J-p+s+3\pi )^2}e^{\mathbf{i}Jy},&\\
&\tilde{\phi }_{21}(x,y)&\\
&=e^{\frac{\mathbf{i}}{4\pi }(q-t+\pi )x+\frac{\mathbf{i}}{8\pi \tau }(p-s+\pi )^2}\displaystyle{\sum_{K\in \mathbb{Z}}}(-1)^Ke^{-\mathbf{i}(q-t)K-\frac{\mathbf{i}}{8\pi \tau }(x-4\pi K-p+s-\pi )^2}e^{\mathbf{i}Ky},&\\
&\tilde{\phi }_{22}(x,y)&\\
&=-\mathbf{i}e^{\frac{\mathbf{i}}{4\pi }(q-t+\pi )x+\frac{\mathbf{i}}{2}y+\frac{\mathbf{i}}{8\pi \tau }(p-s+\pi )^2-\frac{\mathbf{i}}{2}(q-t)}&\\
&\ \ \ \times \displaystyle{\sum_{L\in \mathbb{Z}}}(-1)^Le^{-\mathbf{i}(q-t)L-\frac{\mathbf{i}}{8\pi \tau }(x-4\pi L-p+s-3\pi )^2}e^{\mathbf{i}Ly}.& 
\end{flalign*}
\end{corollary}
Similarly, for $\phi : C(\psi )\rightarrow E_{(\frac{1}{2},\frac{\eta}{2} )}$, we obtain the followings by setting $u=p+s+\pi $ and $v=q+t+\pi $.
\begin{theo}\label{the4.5}
For $E_{(\frac{1}{2},\frac{\eta}{2} )}$, $\mathrm{dim}G(C(\psi ))=1$ holds if and only if $\eta \equiv \mu +\nu +\pi +\pi \tau \ (\mathrm{mod}\ 2\pi (\mathbb{Z}\oplus \tau \mathbb{Z}))$. 
\end{theo}
\begin{corollary}\label{cor4.6}
Under the condition of Theorem \ref{the4.5}, we can take the base of $G(C(\psi ))$ locally expressed as 
\begin{equation*}
\phi :=\left(\begin{array}{ccc}\phi _{11}&\phi _{12}\\\phi _{21}&\phi _{22}\end{array}\right)=\left(\begin{array}{ccc}\phi _{11}(x,y)&\phi _{12}(x,y)\\\phi _{21}(x,y)&\phi _{22}(x,y)\end{array}\right),
\end{equation*}
\begin{flalign*}
&\phi _{11}(x,y)&\\
&=e^{\frac{\mathbf{i}}{4\pi }(q-t-\pi )x+\frac{\mathbf{i}}{8\pi \tau }(p-s-\pi )^2}\displaystyle{\sum_{M\in \mathbb{Z}}}(-1)^Me^{-\mathbf{i}(q-t)M-\frac{\mathbf{i}}{8\pi \tau }(x-4\pi M-p+s+\pi )^2}e^{\mathbf{i}My},&\\
&\phi _{12}(x,y)&\\
&=\frac{\bar{\tau }-\tau }{2\tau }\mathbf{i}e^{-\frac{\mathbf{i}}{4\pi }(q-t+\pi )x-\frac{\mathbf{i}}{8\pi \tau }(p-s+\pi )^2}\displaystyle{\sum_{N\in \mathbb{Z}}}\displaystyle{\sum_{H\in \mathbb{Z}}}(-1)^{N+H}e^{-\mathbf{i}(q-t)N-\frac{\pi \mathbf{i}}{\tau }(H-2N+\frac{1}{2})^2}&\\
&\ \ \ \times \theta _{\varepsilon }(x-(-2\pi H+p-s))e^{\frac{\mathbf{i}}{8\pi \tau }(x+4\pi N-p+s-\pi )^2}e^{\mathbf{i}Ny},&\\
&\phi _{21}(x,y)&\\
&=\mathbf{i}e^{\frac{\mathbf{i}}{4\pi }(q-t+\pi )x+\frac{\mathbf{i}}{2}y+\frac{\mathbf{i}}{8\pi \tau }(p-s-\pi )^2-\frac{\mathbf{i}}{2}(q-t)}&\\
&\ \ \ \times \displaystyle{\sum_{P\in \mathbb{Z}}}(-1)^Pe^{-\mathbf{i}(q-t)P-\frac{\mathbf{i}}{8\pi \tau }(x-4\pi P-p+s-\pi )^2}e^{\mathbf{i}Py},&\\
&\phi _{22}(x,y)&\\
&=-\frac{\bar{\tau }-\tau }{2\tau }e^{-\frac{\mathbf{i}}{4\pi }(q-t+\pi )x+\frac{\mathbf{i}}{2}y-\frac{\mathbf{i}}{8\pi \tau }(p-s+\pi )^2-\frac{\mathbf{i}}{2}(q-t)}&\\
&\ \ \ \times \displaystyle{\sum_{Q\in \mathbb{Z}}}\displaystyle{\sum_{H\in \mathbb{Z}}}(-1)^{Q+H}e^{-\mathbf{i}(q-t)Q-\frac{\pi \mathbf{i}}{\tau }(H-2Q-\frac{1}{2})^2}&\\
&\ \ \ \times \theta _{\varepsilon }(x-(-2\pi H+p-s))e^{\frac{\mathbf{i}}{8\pi \tau }(x+4\pi Q-p+s+\pi )^2}e^{\mathbf{i}Qy}.&
\end{flalign*}
\end{corollary}
The holomorphic maps obtained by Theorem \ref{the4.3} (Corollary \ref{cor4.4}) and Theorem \ref{the4.5} (Corollary \ref{cor4.6}) actually satisfy $\phi \tilde{\phi }=c_{\tau }I_2\ (c_{\tau }\not=0,\ c_{\tau }\in \mathbb{C})$. Namely, $\frac{1}{\sqrt{c_{\tau }}}\phi $ and $\frac{1}{\sqrt{c_{\tau }}}\tilde{\phi } $ give the isomorphism $C(\psi )\cong E_{(\frac{1}{2},\frac{\eta }{2})}$. In order to show this fact, we propose the following lemmas. Note that $\tilde{\phi }_{11}(x,y)$ can be written as 
\begin{align*}
&\displaystyle{\sum_{H\in \mathbb{Z}}}(-1)^He^{-\frac{\pi \mathbf{i}}{\tau }(H-2I-\frac{1}{2})^2}\theta _{\varepsilon }(x-(-2\pi H+p-s))\\
&=\displaystyle{\sum_{H\in \mathbb{Z}}}(-1)^He^{-\frac{\pi \mathbf{i}}{\tau }(H-2I-\frac{1}{2})^2}\\
&\ \ \ \times \Bigl\{\theta _{\varepsilon }(x-(-2\pi H+p-s))-\theta _{\varepsilon }(x-(-4\pi I+p-s-\pi ))\Bigr\},
\end{align*}
since $\displaystyle{\sum_{H\in \mathbb{Z}}}(-1)^He^{-\frac{\pi \mathbf{i}}{\tau }(H-2I-\frac{1}{2})^2}=0$. Similar facts hold true also for $\tilde{\phi }_{12}(x,y)$, $\phi _{12}(x,y)$ and $\phi _{22}(x,y)$.
\begin{lemma}
Let $\mathbf{a}$ be a integer and $\mathbf{a}\not=0$. Then the following identity holds.
\begin{align}
&\displaystyle{\sum_{k\in \mathbb{Z}}}e^{\frac{2\pi \mathbf{i}k\mathbf{a}}{\tau }}\displaystyle{\sum_{l\in \mathbb{Z}}}(-1)^le^{-\frac{\pi \mathbf{i}}{\tau }(l+\frac{1}{2})^2} \notag \\
&\times \Bigl\{\theta _{\varepsilon }(x-(-2\pi k-2\pi l+p-s))-\theta _{\varepsilon }(x-(-2\pi k+p-s+\pi ))\Bigr\}=0. \label{id}
\end{align}
\end{lemma}
\begin{proof}
The holomorphic map $\tilde{\phi } : E_{(\frac{1}{2},\frac{\eta}{2} )}\rightarrow C(\psi )$ satisfies the following differential equations.
\begin{align*}
&d_0\tilde{\phi }_{11}-\tilde{\phi }_{11}d_{\frac{1}{2}}+\psi \tilde{\phi }_{21}d\bar{z}=0,\\
&d_0\tilde{\phi }_{12}-\tilde{\phi }_{12}d_{\frac{1}{2}}+\psi \tilde{\phi }_{22}d\bar{z}=0,\\
&d_1\tilde{\phi }_{21}-\tilde{\phi }_{21}d_{\frac{1}{2}}=0,\\
&d_1\tilde{\phi }_{22}-\tilde{\phi }_{22}d_{\frac{1}{2}}=0.
\end{align*}
They are expressed locally as follows by using $d_0=2\bar{\partial }-\frac{\mathbf{i}}{\pi (\bar{\tau }-\tau )}\mu d\bar{z}$, $d_1=2\bar{\partial }-\frac{\mathbf{i}}{\pi (\bar{\tau }-\tau )}(x+\nu )d\bar{z}$ and $d_{\frac{1}{2}}=2\bar{\partial }-\frac{\mathbf{i}}{\pi (\bar{\tau }-\tau )}(\frac{1}{2}x+\frac{\eta }{2})d\bar{z}\ (\mu =p+q\tau ,\ \nu =s+t\tau ,\ \eta =u+v\tau ,\ u=p+s+\pi ,\ v=q+t+\pi )$.
\begin{align}
&2\bar{\partial }(\tilde{\phi }_{11})-\frac{\mathbf{i}}{\pi (\bar{\tau }-\tau )}\left(-\frac{1}{2}x+\mu -\frac{\eta }{2}\right)\tilde{\phi }_{11}d\bar{z}+\psi \tilde{\phi }_{21}d\bar{z}=0,\label{d1}\\ 
&2\bar{\partial }(\tilde{\phi }_{12})-\frac{\mathbf{i}}{\pi (\bar{\tau }-\tau )}\left(-\frac{1}{2}x+\mu -\frac{\eta }{2}\right)\tilde{\phi }_{12}d\bar{z}+\psi \tilde{\phi }_{22}d\bar{z}=0,\label{d2}\\ 
&2\bar{\partial }(\tilde{\phi }_{21})-\frac{\mathbf{i}}{\pi (\bar{\tau }-\tau )}\left(\frac{1}{2}x+\nu -\frac{\eta }{2}\right)\tilde{\phi }_{21}d\bar{z}=0,\label{d3}\\ 
&2\bar{\partial }(\tilde{\phi }_{22})-\frac{\mathbf{i}}{\pi (\bar{\tau }-\tau )}\left(\frac{1}{2}x+\nu -\frac{\eta }{2}\right)\tilde{\phi }_{22}d\bar{z}=0.\label{d4}
\end{align}
We obtain the following differential equations by multiplying $\tilde{\phi }_{22}$, $\tilde{\phi }_{21}$, $\tilde{\phi }_{12}$, $\tilde{\phi }_{11}$ to the differential equations (\ref{d1}), (\ref{d2}), (\ref{d3}), (\ref{d4}), respectively.
\begin{align}
&\tilde{\phi }_{22}2\bar{\partial }(\tilde{\phi }_{11})-\frac{\mathbf{i}}{\pi (\bar{\tau }-\tau )}\left(-\frac{1}{2}x+\mu -\frac{\eta }{2}\right)\tilde{\phi }_{11}\tilde{\phi }_{22}d\bar{z}+\psi \tilde{\phi }_{21}\tilde{\phi }_{22}d\bar{z}=0,\label{d5}\\ 
&\tilde{\phi }_{21}2\bar{\partial }(\tilde{\phi }_{12})-\frac{\mathbf{i}}{\pi (\bar{\tau }-\tau )}\left(-\frac{1}{2}x+\mu -\frac{\eta }{2}\right)\tilde{\phi }_{12}\tilde{\phi }_{21}d\bar{z}+\psi \tilde{\phi }_{21}\tilde{\phi }_{22}d\bar{z}=0,\label{d6}\\ 
&\tilde{\phi }_{12}2\bar{\partial }(\tilde{\phi }_{21})-\frac{\mathbf{i}}{\pi (\bar{\tau }-\tau )}\left(\frac{1}{2}x+\nu -\frac{\eta }{2}\right)\tilde{\phi }_{12}\tilde{\phi }_{21}d\bar{z}=0,\label{d7}\\ 
&\tilde{\phi }_{11}2\bar{\partial }(\tilde{\phi }_{22})-\frac{\mathbf{i}}{\pi (\bar{\tau }-\tau )}\left(\frac{1}{2}x+\nu -\frac{\eta }{2}\right)\tilde{\phi }_{11}\tilde{\phi }_{22}d\bar{z}=0.\label{d8}
\end{align}
The sum $(\ref{d5})+(\ref{d8})-(\ref{d6})-(\ref{d7})$ turns out to be the following differential equation.
\begin{equation}
2\bar{\partial }(\tilde{\phi }_{11}\tilde{\phi }_{22}-\tilde{\phi }_{12}\tilde{\phi }_{21})+\frac{\mathbf{i}}{\pi (\bar{\tau }-\tau )}(\pi +\pi \tau )(\tilde{\phi }_{11}\tilde{\phi }_{22}-\tilde{\phi }_{12}\tilde{\phi }_{21})d\bar{z}=0. \label{diff}
\end{equation}
We Fourier-expand $\tilde{\phi }_{11}$, $\tilde{\phi }_{12}$, $\tilde{\phi }_{21}$, $\tilde{\phi }_{22}$ as 
 \begin{align*}
&\tilde{\phi }_{11}=\displaystyle{\sum_{I\in \mathbb{Z}}}\tilde{\phi }_{11,I}(x)e^{\mathbf{i}Iy},\ \tilde{\phi }_{12}=e^{\frac{\mathbf{i}}{2}y}\displaystyle{\sum_{J\in \mathbb{Z}}}\tilde{\phi }_{12,J}(x)e^{\mathbf{i}Jy},\\
&\tilde{\phi }_{21}=\displaystyle{\sum_{K\in \mathbb{Z}}}\tilde{\phi }_{21,K}(x)e^{\mathbf{i}Ky},\ \tilde{\phi }_{22}=e^{\frac{\mathbf{i}}{2}y}\displaystyle{\sum_{L\in \mathbb{Z}}}\tilde{\phi }_{22,L}(x)e^{\mathbf{i}Ly},
\end{align*}
and calculate $\mathrm{det}\tilde{\phi }=\tilde{\phi }_{11}\tilde{\phi }_{22}-\tilde{\phi }_{12}\tilde{\phi }_{21}$.
\begin{flalign*}
&\mathrm{det}\tilde{\phi }&\\
&=\Bigl(\displaystyle{\sum_{I\in \mathbb{Z}}}\tilde{\phi }_{11,I}(x)e^{\mathbf{i}Iy}\Bigr)\Bigl(e^{\frac{\mathbf{i}}{2}y}\displaystyle{\sum_{J\in \mathbb{Z}}}\tilde{\phi }_{12,J}(x)e^{\mathbf{i}Jy}\Bigr)&\\
&\ \ \ -\Bigl(\displaystyle{\sum_{K\in \mathbb{Z}}}\tilde{\phi }_{21,K}(x)e^{\mathbf{i}Ky}\Bigr)\Bigl(e^{\frac{\mathbf{i}}{2}y}\displaystyle{\sum_{L\in \mathbb{Z}}}\tilde{\phi }_{22,L}(x)e^{\mathbf{i}Ly}\Bigr)&\\
&=e^{\frac{\mathbf{i}}{2}y}\displaystyle{\sum_{I\in \mathbb{Z}}\sum_{L\in \mathbb{Z}}}\tilde{\phi }_{11,I}(x)\tilde{\phi }_{22,L}(x)e^{\mathbf{i}(I+L)y}-e^{\frac{\mathbf{i}}{2}y}\displaystyle{\sum_{J\in \mathbb{Z}}\sum_{K\in \mathbb{Z}}}\tilde{\phi }_{12,J}(x)\tilde{\phi }_{21,K}(x)e^{\mathbf{i}(J+K)y}&\\
&=e^{\frac{\mathbf{i}}{2}y}\displaystyle{\sum_{\mathbf{a}\in \mathbb{Z}}\sum_{I\in \mathbb{Z}}}\tilde{\phi }_{11,I}(x)\tilde{\phi }_{22,\mathbf{a} -I}(x)e^{\mathbf{i}\mathbf{a}y}-e^{\frac{\mathbf{i}}{2}y}\displaystyle{\sum_{\mathbf{a}\in \mathbb{Z}}\sum_{J\in \mathbb{Z}}}\tilde{\phi }_{12,J}(x)\tilde{\phi }_{21,\mathbf{a}-J}(x)e^{\mathbf{i}\mathbf{a}y}&\\
&=e^{\frac{\mathbf{i}}{2}y}\displaystyle{\sum_{\mathbf{a}\in \mathbb{Z}}}\Bigl\{\displaystyle{\sum_{I\in \mathbb{Z}}}\tilde{\phi }_{11,I}(x)\tilde{\phi }_{22,\mathbf{a}-I}(x)-\displaystyle{\sum_{J\in \mathbb{Z}}}\tilde{\phi }_{12,J}(x)\tilde{\phi }_{21,\mathbf{a}-J}(x)\Bigr\}e^{\mathbf{i}\mathbf{a}y}.&
\end{flalign*}
In the third equality, we put $I+L=J+K=\mathbf{a}$. Here, we define the function $\tilde{\Phi }_{\mathbf{a}}(x)$ as follows.
\begin{equation*}
\tilde{\Phi }_{\mathbf{a} }(x):=\displaystyle{\sum_{I\in \mathbb{Z}}}\tilde{\phi }_{11,I}(x)\tilde{\phi }_{22,\mathbf{a} -I}(x)-\displaystyle{\sum_{J\in \mathbb{Z}}}\tilde{\phi }_{12,J}(x)\tilde{\phi }_{21,\mathbf{a} -J}(x).
\end{equation*}
Namely, $\mathrm{det}\tilde{\phi }$ is given by $e^{\frac{\mathbf{i}}{2}y}\displaystyle{\sum_{\mathbf{a} \in \mathbb{Z}}}\tilde{\Phi }_{\mathbf{a} }(x)e^{\mathbf{i}\mathbf{a} y}$. By substituting $\tilde{\phi }_{11}\tilde{\phi }_{22}-\tilde{\phi }_{12}\tilde{\phi }_{21}=e^{\frac{\mathbf{i}}{2}y}\displaystyle{\sum_{\mathbf{a} \in \mathbb{Z}}}\tilde{\Phi }_{\mathbf{a} }(x)e^{\mathbf{i}\mathbf{a} y}$ into the differential equation (\ref{diff}), one obtains
\begin{equation*}
-\frac{2\tau }{\bar{\tau }-\tau }e^{\frac{\mathbf{i}}{2}y}\displaystyle{\sum_{\mathbf{a} \in \mathbb{Z}}}\Bigl\{\frac{d}{dx}\tilde{\Phi }_{\mathbf{a}}(x)-\frac{\mathbf{i}}{\tau }\left(\mathbf{a} +1+\frac{\tau }{2}\right)\tilde{\Phi }_{\mathbf{a} }(x)\Bigr\}e^{\mathbf{i}\mathbf{a} y}d\bar{z}=0.
\end{equation*}
Thus, we solve the differential equation
\begin{equation*}
\frac{d}{dx}\tilde{\Phi }_{\mathbf{a} }(x)-\frac{\mathbf{i}}{\tau }\left(\mathbf{a}+1+\frac{\tau }{2}\right)\tilde{\Phi }_{\mathbf{a} }(x)=0.
\end{equation*}
The general solution of this differential equation is expressed as
\begin{equation*}
\tilde{\Phi }_{\mathbf{a}}(x)=\tilde{C}_{\mathbf{a}}e^{\frac{\mathbf{i}}{\tau }(\mathbf{a}+1+\frac{\tau }{2})x},
\end{equation*}
where $\tilde{C}_{\mathbf{a}}$ is an arbitrary constant. Recall that $\tilde{\phi }_{11,I}(x)$, $\tilde{\phi }_{12,J}(x)$, $\tilde{\phi }_{21,K}(x)$ and $\tilde{\phi }_{22,L}(x)$ satisfy the relations
\begin{align*}
&\tilde{\phi }_{11,I}(x+2\pi )=\tilde{\phi }_{12,I}(x),\ \tilde{\phi }_{12,J}(x+2\pi )=\tilde{\phi }_{11,J+1}(x),\\
&\tilde{\phi }_{21,K}(x+2\pi )=\tilde{\phi }_{22,K-1}(x),\ \tilde{\phi }_{22,L}(x+2\pi )=\tilde{\phi }_{21,L}(x). 
\end{align*}
Therefore, 
\begin{align*}
\tilde{\Phi }_{\mathbf{a}}(x+2\pi )&=\displaystyle{\sum_{I\in \mathbb{Z}}}\tilde{\phi }_{11,I}(x+2\pi )\tilde{\phi }_{22,\mathbf{a}-I}(x+2\pi )-\displaystyle{\sum_{J\in \mathbb{Z}}}\tilde{\phi }_{12,J}(x+2\pi )\tilde{\phi }_{21,\mathbf{a}-J}(x+2\pi )\\
&=\displaystyle{\sum_{I\in \mathbb{Z}}}\tilde{\phi }_{12,I}(x)\tilde{\phi }_{21,\mathbf{a} -I}(x)-\displaystyle{\sum_{J\in \mathbb{Z}}}\tilde{\phi }_{11,J+1}(x)\tilde{\phi }_{22,\mathbf{a} -J-1}(x)\\
&=\displaystyle{\sum_{I\in \mathbb{Z}}}\tilde{\phi }_{12,I}(x)\tilde{\phi }_{21,\mathbf{a} -I}(x)-\displaystyle{\sum_{J\in \mathbb{Z}}}\tilde{\phi }_{11,J}(x)\tilde{\phi }_{22,\mathbf{a} -J}(x)\\
&=-\Bigl(\displaystyle{\sum_{J\in \mathbb{Z}}}\tilde{\phi }_{11,J}(x)\tilde{\phi }_{22,\mathbf{a} -J}(x)-\displaystyle{\sum_{I\in \mathbb{Z}}}\tilde{\phi }_{12,I}(x)\tilde{\phi }_{21,\mathbf{a} -I}(x)\Bigr)\\
&=-\tilde{\Phi }_{\mathbf{a}}(x).
\end{align*}
Since
\begin{align*}
\tilde{\Phi }_{\mathbf{a}}(x+2\pi )&=\tilde{C}_{\mathbf{a}}e^{\frac{\mathbf{i}}{\tau }(\mathbf{a}+1+\frac{\tau }{2})(x+2\pi )}\\
&=\tilde{C}_{\mathbf{a}}e^{\frac{\mathbf{i}}{\tau }(\mathbf{a}+1+\frac{\tau }{2})x}(-e^{\frac{\mathbf{i}}{\tau }(2\pi \mathbf{a}+2\pi )})\\
&=-e^{\frac{2\pi \mathbf{i}}{\tau }(\mathbf{a}+1)}\tilde{\Phi }_{\mathbf{a}}(x),
\end{align*}
one has $e^{\frac{2\pi \mathbf{i}}{\tau }(\mathbf{a}+1)}=1$, namely, $\mathbf{a}=-1$. Therefore, $\tilde{C}_{\mathbf{a}}$ satisfies $\tilde{C}_{\mathbf{a}}=0\ (\mathbf{a}\not=-1)$. By Corollary \ref{cor4.4}, $\mathrm{det}\tilde{\phi }$ is expressed locally as
\begin{align*}
&\mathrm{det}\tilde{\phi }=\frac{\bar{\tau }-\tau }{2\tau }e^{\frac{\mathbf{i}}{2}y-\frac{\mathbf{i}}{2\tau }(p-s+2\pi )-\frac{\mathbf{i}}{2}(q-t)}\displaystyle{\sum_{\mathbf{a}\in \mathbb{Z}}}\displaystyle{\sum_{k\in \mathbb{Z}}}\displaystyle{\sum_{l\in \mathbb{Z}}}(-1)^{\mathbf{a}+l}e^{-\mathbf{i}(q-t)\mathbf{a}-\frac{\pi \mathbf{i}}{\tau }(l+\frac{1}{2})^2}&\\
&\ \ \ \ \ \ \ \ \ \times \Bigl\{\theta _{\varepsilon }(x-(-2\pi k-2\pi l+p-s))-\theta _{\varepsilon }(x-(-2\pi k+p-s+\pi ))\Bigr\}&\\
&\ \ \ \ \ \ \ \ \ \times e^{\frac{\mathbf{i}}{\tau }((\mathbf{a}+1+\frac{\tau }{2})x-2\pi \mathbf{a}^2+2\pi k\mathbf{a}-(p-s+5\pi )\mathbf{a}+2\pi k-2\pi )}e^{\mathbf{i}\mathbf{a}y},&\\
&\tilde{C}_{\mathbf{a}}=\frac{\bar{\tau }-\tau }{2\tau }e^{-\frac{\mathbf{i}}{2\tau }(p-s+2\pi )-\frac{\mathbf{i}}{2}(q-t)}\displaystyle{\sum_{k\in \mathbb{Z}}}\displaystyle{\sum_{l\in \mathbb{Z}}}(-1)^{\mathbf{a}+l}e^{-\mathbf{i}(q-t)\mathbf{a}-\frac{\pi \mathbf{i}}{\tau }(l+\frac{1}{2})^2}&\\
&\ \ \ \ \ \ \ \ \ \times \Bigl\{\theta _{\varepsilon }(x-(-2\pi k-2\pi l+p-s))-\theta _{\varepsilon }(x-(-2\pi k+p-s+\pi ))\Bigr\}&\\
&\ \ \ \ \ \ \ \ \ \times e^{\frac{\mathbf{i}}{\tau }(-2\pi \mathbf{a}^2+2\pi k\mathbf{a}-(p-s+5\pi )\mathbf{a}+2\pi k-2\pi )}.
\end{align*}
Clearly, the formula $\mathrm{det}\tilde{\phi }=e^{\frac{\mathbf{i}}{2}y}\displaystyle{\sum_{\mathbf{a}\in \mathbb{Z}}}\tilde{C}_{\mathbf{a}}e^{\frac{\mathbf{i}}{\tau }(\mathbf{a}+1+\frac{\tau }{2})x}e^{\mathbf{i}\mathbf{a}y},\ \tilde{C}_{\mathbf{a}}=0\ (\mathbf{a}\not=-1)$ is equivalent to the formula $\mathrm{det}\tilde{\phi }=e^{\frac{\mathbf{i}}{2}y}\displaystyle{\sum_{\mathbf{a}\in \mathbb{Z}}}\tilde{C}_{\mathbf{a}-1}e^{\frac{\mathbf{i}}{\tau }(\mathbf{a}+\frac{\tau }{2})x}e^{\mathbf{i}(\mathbf{a}-1)y},\ \tilde{C}_{\mathbf{a}-1}=0\ (\mathbf{a}\not=0)$, and $\tilde{C}_{\mathbf{a}-1}$ is given by
\begin{align*}
&\tilde{C}_{\mathbf{a}-1}=-\frac{\bar{\tau }-\tau }{2\tau }(-1)^{\mathbf{a}}e^{\frac{\mathbf{i}}{2\tau }(p-s)+\frac{\mathbf{i}}{2}(q-t)+\frac{\mathbf{i}}{\tau }(-2\pi \mathbf{a}^2-(p-s+\pi )\mathbf{a})-\mathbf{i}(q-t)\mathbf{a}}&\\
&\ \ \ \ \ \ \ \ \ \ \times \displaystyle{\sum_{k\in \mathbb{Z}}}e^{\frac{2\pi \mathbf{i}k\mathbf{a}}{\tau }}\displaystyle{\sum_{l\in \mathbb{Z}}}(-1)^le^{-\frac{\pi \mathbf{i}}{\tau }(l+\frac{1}{2})^2}&\\
&\ \ \ \ \ \ \ \ \ \ \times \Bigl\{\theta _{\varepsilon }(x-(-2\pi k-2\pi l+p-s))-\theta _{\varepsilon }(x-(-2\pi k+p-s+\pi ))\Bigr\}.&
\end{align*}
Thus, the condition $\tilde{C}_{\mathbf{a}-1}=0\ (\mathbf{a}\not=0)$ implies the identity (\ref{id}).
\end{proof}
\begin{lemma}\label{lem4.8}
$\displaystyle{\sum_{l\in \mathbb{Z}}}(-1)^l(2l+1)e^{-\frac{\pi \mathbf{i}}{\tau }(l+\frac{1}{2})^2}\not=0.$
\end{lemma}
\begin{proof}
We consider the following theta function, where $\beta \in \mathbb{C}$.
\begin{equation*}
\vartheta _{\frac{1}{2},\frac{1}{2}}\left(\beta ,-\frac{1}{\tau } \right)=\displaystyle{\sum_{l\in \mathbb{Z}}}e^{\pi \mathbf{i}(l+\frac{1}{2})^2(-\frac{1}{\tau }) +2\pi \mathbf{i}(l+\frac{1}{2})(\beta +\frac{1}{2})}.
\end{equation*}
We differentiate it with respect to $\beta $.
\begin{align*}
\frac{\partial }{\partial \beta }\vartheta _{\frac{1}{2},\frac{1}{2}}\left(\beta ,-\frac{1}{\tau } \right)&=\displaystyle{\sum_{l\in \mathbb{Z}}}2\pi \mathbf{i}\left(l+\frac{1}{2}\right)e^{\pi \mathbf{i}(l+\frac{1}{2})^2(-\frac{1}{\tau }) +2\pi \mathbf{i}(l+\frac{1}{2})(\beta +\frac{1}{2})}\\
&=\pi \mathbf{i}\displaystyle{\sum_{l\in \mathbb{Z}}}(2l+1)e^{\pi \mathbf{i}(l+\frac{1}{2})^2(-\frac{1}{\tau }) +2\pi \mathbf{i}(l+\frac{1}{2})(\beta +\frac{1}{2})}.
\end{align*}
Hence, it can be seen that 
\begin{align*}
\left.\frac{\partial }{\partial \beta }\vartheta _{\frac{1}{2},\frac{1}{2}}\left(\beta ,-\frac{1}{\tau } \right)\right|_{\beta =0}&=\pi \mathbf{i}\displaystyle{\sum_{l\in \mathbb{Z}}}(2l+1)e^{\pi \mathbf{i}(l+\frac{1}{2})^2(-\frac{1}{\tau }) +\pi \mathbf{i}(l+\frac{1}{2})}\\
&=-\pi \displaystyle{\sum_{l\in \mathbb{Z}}}(-1)^l(2l+1)e^{\pi \mathbf{i}(l+\frac{1}{2})^2(-\frac{1}{\tau }) }.
\end{align*}
Thus, we can prove this lemma if we show $\left.\frac{\partial }{\partial \beta }\vartheta _{\frac{1}{2},\frac{1}{2}}(\beta ,-\frac{1}{\tau } )\right|_{\beta =0}\not=0$. By the Jacobi's differential formula (see \cite{t}, $\mathrm{p}$.64),
\begin{equation*}
\left.\frac{\partial }{\partial \beta }\vartheta _{\frac{1}{2},\frac{1}{2}}\left(\beta ,-\frac{1}{\tau } \right)\right|_{\beta =0}=-\pi \vartheta _{0,0}\left(0,-\frac{1}{\tau } \right)\vartheta _{0,\frac{1}{2}}\left(0,-\frac{1}{\tau } \right)\vartheta _{\frac{1}{2},0}\left(0,-\frac{1}{\tau } \right)
\end{equation*}
holds. Since the zeros of $\vartheta _{0,0}(\beta ,\tau )$, $\vartheta _{0,\frac{1}{2}}(\beta ,\tau )$ and $\vartheta _{\frac{1}{2},0}(\beta ,\tau )$ are 
\begin{align*}
&\text{The zeros of}\ \vartheta _{0,0}(\beta ,\tau )\ \ \ :\ \ \ \biggl\{\left(\left(p+\frac{1}{2}\right)\tau +\left(q+\frac{1}{2}\right),\tau \right)\in \mathbb{C}\times \mathbb{H}\biggr\},\\
&\text{The zeros of}\ \vartheta _{0,\frac{1}{2}}(\beta ,\tau )\ \ \ :\ \ \ \biggl\{\left(\left(p+\frac{1}{2}\right)\tau +q,\tau \right)\in \mathbb{C}\times \mathbb{H}\biggr\},\\
&\text{The zeros of}\ \vartheta _{\frac{1}{2},0}(\beta ,\tau )\ \ \ :\ \ \ \biggl\{\left(p\tau +\left(q+\frac{1}{2}\right),\tau \right)\in \mathbb{C}\times \mathbb{H}\biggr\},
\end{align*}
where $p$, $q\in \mathbb{Z}$, we see $\vartheta _{0,0}(0,-\frac{1}{\tau })\not=0$, $\vartheta _{0,\frac{1}{2}}(0,-\frac{1}{\tau })\not=0$ and $\vartheta _{\frac{1}{2},0}(0,-\frac{1}{\tau })\not=0$, which implies $\left.\frac{\partial }{\partial \beta }\vartheta _{\frac{1}{2},\frac{1}{2}}(\beta ,-\frac{1}{\tau } )\right|_{\beta =0}\not=0$.
\end{proof}
Now, we consider the composition $\phi \tilde{\phi } : E_{(\frac{1}{2},\frac{\eta }{2})}\rightarrow E_{(\frac{1}{2},\frac{\eta }{2})}$. By Corollary \ref{cor3.2}, the diagonal components of $\phi \tilde{\phi }$ are same, and the off-diagonal components are 0. Thus, we calculate the (1,1) component of $\phi \tilde{\phi } $. Using the identity $(\ref{id})$, it turns out to be
\begin{flalign*}
&\phi _{11}(x,y)\tilde{\phi }_{11}(x,y)+\phi _{12}(x,y)\tilde{\phi }_{21}(x,y)&\\
&=\biggl(e^{\frac{\mathbf{i}}{4\pi }(q-t-\pi )x+\frac{\mathbf{i}}{8\pi \tau }(p-s-\pi )^2}\displaystyle{\sum_{M\in \mathbb{Z}}}(-1)^Me^{-\mathbf{i}(q-t)M-\frac{\mathbf{i}}{8\pi \tau }(x-4\pi M-p+s+\pi )^2}e^{\mathbf{i}My}\biggr)&\\
&\ \ \ \times \biggl(-\frac{\bar{\tau }-\tau }{2\tau }\mathbf{i}e^{-\frac{\mathbf{i}}{4\pi }(q-t-\pi )x-\frac{\mathbf{i}}{8\pi \tau }(p-s-\pi )^2}\displaystyle{\sum_{I\in \mathbb{Z}}}\displaystyle{\sum_{H\in \mathbb{Z}}}(-1)^{I+H}e^{-\mathbf{i}(q-t)I-\frac{\pi \mathbf{i}}{\tau }(H-2I-\frac{1}{2})^2}&\\
&\ \ \ \times \Bigl\{\theta _{\varepsilon }(x-(-2\pi H+p-s))-\vartheta _{\varepsilon }(x-(-4\pi I+p-s-\pi ))\Bigr\}e^{\frac{\mathbf{i}}{8\pi \tau }(x+4\pi I-p+s+\pi )^2}e^{\mathbf{i}Iy}\biggr)&\\
&\ \ \ +\biggl(\frac{\bar{\tau }-\tau }{2\tau }\mathbf{i}e^{-\frac{\mathbf{i}}{4\pi }(q-t+\pi )x-\frac{\mathbf{i}}{8\pi \tau }(p-s+\pi )^2}\displaystyle{\sum_{N\in \mathbb{Z}}}\displaystyle{\sum_{H\in \mathbb{Z}}}(-1)^{N+H}e^{-\mathbf{i}(q-t)N-\frac{\pi \mathbf{i}}{\tau }(H-2N+\frac{1}{2})^2}&\\
&\ \ \ \times \Bigl\{\theta _{\varepsilon }(x-(-2\pi H+p-s))-\vartheta _{\varepsilon }(x-(-4\pi N+p-s+\pi ))\Bigr\}e^{\frac{\mathbf{i}}{8\pi \tau }(x+4\pi N-p+s-\pi )^2}e^{\mathbf{i}Ny}\biggr)&\\
&\ \ \ \times \biggl(e^{\frac{\mathbf{i}}{4\pi }(q-t+\pi )x+\frac{\mathbf{i}}{8\pi \tau }(p-s+\pi )^2}\displaystyle{\sum_{K\in \mathbb{Z}}}(-1)^Ke^{-\mathbf{i}(q-t)K-\frac{\mathbf{i}}{8\pi \tau }(x-4\pi K-p+s-\pi )^2}e^{\mathbf{i}Ky}\biggr)&\\
&=-\frac{\bar{\tau }-\tau }{2\tau }\mathbf{i}\displaystyle{\sum_{\mathbf{a}\in \mathbb{Z}}}\displaystyle{\sum_{I\in \mathbb{Z}}}\displaystyle{\sum_{H\in \mathbb{Z}}}(-1)^{\mathbf{a}+H}e^{-\mathbf{i}(q-t)\mathbf{a}-\frac{\pi \mathbf{i}}{\tau }(H-2I-\frac{1}{2})^2}&\\
&\ \ \ \times \Bigl\{\theta _{\varepsilon }(x-(-2\pi H+p-s))-\theta _{\varepsilon }(x-(-4\pi I+p-s-\pi ))\Bigr\}&\\
&\ \ \ \times e^{\frac{\mathbf{i}}{\tau }(\mathbf{a}x-2\pi \mathbf{a}^2+4\pi I\mathbf{a}+(-p+s+\pi )\mathbf{a})}e^{\mathbf{i}\mathbf{a}y}&\\
&\ \ \ +\frac{\bar{\tau }-\tau }{2\tau }\mathbf{i}\displaystyle{\sum_{\mathbf{a}\in \mathbb{Z}}}\displaystyle{\sum_{N\in \mathbb{Z}}}\displaystyle{\sum_{H\in \mathbb{Z}}}(-1)^{\mathbf{a}+H}e^{-\mathbf{i}(q-t)\mathbf{a}-\frac{\pi \mathbf{i}}{\tau }(H-2N+\frac{1}{2})^2}&\\
&\ \ \ \times \Bigl\{\theta _{\varepsilon }(x-(-2\pi H+p-s))-\theta _{\varepsilon }(x-(-4\pi N+p-s+\pi ))\Bigr\}&\\
&\ \ \ \times e^{\frac{\mathbf{i}}{\tau }(\mathbf{a}x-2\pi \mathbf{a}^2+4\pi N\mathbf{a}+(-p+s-\pi )\mathbf{a})}e^{\mathbf{i}\mathbf{a}y}&\\
&=\frac{\bar{\tau }-\tau }{2\tau }\mathbf{i}\displaystyle{\sum_{\mathbf{a}\in \mathbb{Z}}}\displaystyle{\sum_{k\in \mathbb{Z}}}\displaystyle{\sum_{l\in \mathbb{Z}}}(-1)^{\mathbf{a}+l}e^{-\mathbf{i}(q-t)\mathbf{a}-\frac{\pi \mathbf{i}}{\tau }(l+\frac{1}{2})^2}&\\
&\ \ \ \times \Bigl\{\theta _{\varepsilon }(x-(-2\pi k-2\pi l+p-s))-\theta _{\varepsilon }(x-(-2\pi k+p-s+\pi ))\Bigr\}&\\
&\ \ \ \times e^{\frac{\mathbf{i}}{\tau }(\mathbf{a}x-2\pi \mathbf{a}^2+2\pi k\mathbf{a}+(-p+s-\pi )\mathbf{a})}e^{\mathbf{i}\mathbf{a}y}&\\
&=\frac{\bar{\tau }-\tau }{2\tau }\mathbf{i}\displaystyle{\sum_{k\in \mathbb{Z}}}\displaystyle{\sum_{l\in \mathbb{Z}}}(-1)^le^{-\frac{\pi \mathbf{i}}{\tau }(l+\frac{1}{2})^2}&\\
&\ \ \ \times \Bigl\{\theta _{\varepsilon }(x-(-2\pi k-2\pi l+p-s))-\theta _{\varepsilon }(x-(-2\pi k+p-s+\pi ))\Bigr\}&\\
&=\frac{\bar{\tau }-\tau }{4\tau }\mathbf{i}\displaystyle{\sum_{l\in \mathbb{Z}}}(-1)^l(2l+1)e^{-\frac{\pi \mathbf{i}}{\tau }(l+\frac{1}{2})^2}.
\end{flalign*}
This value is not zero by Lemma \ref{lem4.8}. We denote this value by $c_{\tau }$. Thus, we obtain the following theorem.
\begin{theo}\label{the4.9}
The mapping cone $C(\psi )$ is isomorphic to $E_{(\frac{1}{2},\frac{\eta }{2})}$ with $\eta \equiv \mu +\nu +\pi +\pi \tau \ (\mathrm{mod}\ 2\pi (\mathbb{Z}\oplus \tau \mathbb{Z}))$, where $\frac{1}{\sqrt{c_{\tau }}}\phi $ and $\frac{1}{\sqrt{c_{\tau }}}\tilde{\phi } $ give the isomorphism. 
\end{theo}

\section{Geometric interpretation}
In this section, we discuss a geometric interpretation of the mapping cone $C(\psi )$ from the viewpoint of the corresponding symplectic geometry. First we recall the Fukaya category \cite{Fuk} Fuk($M$) following \cite{P}.

Let $(M,\check{M})$ be a mirror pair. The objects of Fuk($M$) are special Lagrangian submanifolds of $M$ endowed with flat local systems. We denote by $\mathcal{U}_i=(\mathcal{L}_i,\mathcal{E}_i,\nabla_{\mathcal{E}_i})$ an object of Fuk($M$) where $\mathcal{L}_i$ denotes a special Lagrangian submanifold and $(\mathcal{E}_i,\nabla_{\mathcal{E}_i})$ denotes a local system. The space of morphisms $\mathcal{C}(\mathcal{U}_i,\mathcal{U}_j)$ are defined as 
\begin{equation*}
\mathcal{C}(\mathcal{U}_i,\mathcal{U}_j):=\mathbb{C}^{\sharp \{\mathcal{L}_i\cap \mathcal{L}_j\}}\otimes \mathrm{Hom}(\mathcal{E}_i,\mathcal{E}_j),
\end{equation*}
where the $\mathrm{Hom}$ represents homomorphisms of vector spaces at the points of intersection. There is a $\mathbb{Z}$-grading on the morphisms by considering the Maslov index at the points of intersection. The degree $r$ part is denoted $\mathcal{C}^r(\mathcal{U}_i,\mathcal{U}_j)$. The $A_{\infty }$ structure $\{m_k\}$ in Fuk($M$) is given by summing over holomorphic maps from the disk $D^2$, which take the components of the boundary $S^1=\partial D^2$ to the special Lagrangian objects. An element $u_j$ of $\mathcal{C}(\mathcal{U}_j,\mathcal{U}_{j+1})$ is represented by a pair $u_j=t_j\cdot a_j$, where $a_j\in \mathcal{L}_j\cap \mathcal{L}_{j+1}$, and $t_j$ is a matrix in $\mathrm{Hom}(\mathcal{E}_j|_{a_j},\mathcal{E}_{j+1}|_{a_j})$. Here, we assume that $\mathcal{L}_j$ and $\mathcal{L}_{j+1}$ are transversal to each other ($j=1,2,\cdots,k$), and $\mathcal{L}_1$ and $\mathcal{L}_{k+1}$ are also transversal to each other. The composition map $m_k : \mathcal{C}(\mathcal{U}_1,\mathcal{U}_2)\otimes \cdots \otimes \mathcal{C}(\mathcal{U}_k,\mathcal{U}_{k+1})\rightarrow \mathcal{C}(\mathcal{U}_1,\mathcal{U}_{k+1})$ is defined by
\begin{equation*}
m_k(u_1\otimes \cdots \otimes u_k):=\displaystyle{\sum_{a_{k+1}\in \mathcal{L}_1\cap \mathcal{L}_{k+1}}}c(u_1, \cdots ,u_k, a_{k+1})\cdot a_{k+1},
\end{equation*}
where (notation explained below)
\begin{equation*}
c(u_1, \cdots ,u_k, a_{k+1})=\displaystyle{\sum_{\phi }}\pm e^{2\pi \mathbf{i}\int\phi ^*\omega }\cdot Pe^{\oint\phi ^*\beta }
\end{equation*}
is a matrix in $\mathrm{Hom}(\mathcal{E}_1|_{a_{k+1}},\mathcal{E}_{k+1}|_{a_{k+1}})$. Here we sum over pseudo holomorphic maps $\phi : D^2\rightarrow M$, up to equivalence, with the following conditions along the boundary : there are $k+1$ points $p_j=e^{2\pi \mathbf{i}\alpha _j}$ such that $\phi (p_j)=a_j$ and $\phi (e^{2\pi \mathbf{i}\alpha })\in \mathcal{L}_j$ for $\alpha \in (\alpha _{j-1},\alpha _j)$ (see also Figure 2). In the above, $\omega $ is the complexified K\"{a}hlar form, and $P$ represents a path-ordered integration, where $\beta $ is the connection of the flat bundle along the local system on the boundary (we omit the expression of the sign). The path-ordered integration is defined by
\begin{equation*}
Pe^{\oint\phi ^*\beta }:=Pe^{\int_{\alpha _k}^{\alpha _{k+1}}\phi ^*\beta _{k+1}d\alpha }\cdot t_k\cdot Pe^{\int_{\alpha _{k-1}}^{\alpha _k}\phi ^*\beta _kd\alpha }\cdot t_{k-1}\cdot \cdots \cdot t_1\cdot Pe^{\int_{\alpha _{k+1}}^{\alpha _1}\phi ^*\beta _1d\alpha }.
\end{equation*}
\begin{figure}
\begin{center}
{\unitlength 0.1in%
\begin{picture}(46.3500,23.4000)(1.9500,-33.4500)%
%
\special{pn 8}%
\special{ar 1260 2410 656 656 0.0000000 6.2831853}%
%
\special{pn 8}%
\special{pa 3190 1930}%
\special{pa 3220 1918}%
\special{pa 3249 1907}%
\special{pa 3279 1895}%
\special{pa 3308 1883}%
\special{pa 3338 1870}%
\special{pa 3396 1842}%
\special{pa 3425 1827}%
\special{pa 3453 1811}%
\special{pa 3509 1777}%
\special{pa 3563 1739}%
\special{pa 3589 1719}%
\special{pa 3615 1698}%
\special{pa 3640 1676}%
\special{pa 3664 1654}%
\special{pa 3688 1631}%
\special{pa 3710 1607}%
\special{pa 3733 1583}%
\special{pa 3754 1558}%
\special{pa 3774 1532}%
\special{pa 3793 1506}%
\special{pa 3812 1479}%
\special{pa 3829 1452}%
\special{pa 3845 1424}%
\special{pa 3860 1395}%
\special{pa 3874 1367}%
\special{pa 3887 1337}%
\special{pa 3899 1307}%
\special{pa 3909 1277}%
\special{pa 3918 1247}%
\special{pa 3936 1185}%
\special{pa 3940 1170}%
\special{fp}%
%
\special{pn 8}%
\special{pa 3450 1600}%
\special{pa 3449 1632}%
\special{pa 3449 1664}%
\special{pa 3447 1728}%
\special{pa 3445 1760}%
\special{pa 3444 1792}%
\special{pa 3441 1823}%
\special{pa 3439 1855}%
\special{pa 3436 1887}%
\special{pa 3428 1951}%
\special{pa 3423 1983}%
\special{pa 3411 2047}%
\special{pa 3403 2078}%
\special{pa 3395 2110}%
\special{pa 3386 2141}%
\special{pa 3376 2172}%
\special{pa 3365 2203}%
\special{pa 3353 2233}%
\special{pa 3340 2263}%
\special{pa 3326 2292}%
\special{pa 3311 2321}%
\special{pa 3295 2349}%
\special{pa 3278 2376}%
\special{pa 3259 2402}%
\special{pa 3240 2427}%
\special{pa 3220 2452}%
\special{pa 3199 2475}%
\special{pa 3176 2498}%
\special{pa 3153 2520}%
\special{pa 3105 2562}%
\special{pa 3079 2582}%
\special{pa 3054 2602}%
\special{pa 3028 2621}%
\special{pa 3002 2641}%
\special{pa 2990 2650}%
\special{fp}%
%
\special{pn 8}%
\special{pa 3030 2410}%
\special{pa 3059 2423}%
\special{pa 3088 2437}%
\special{pa 3118 2450}%
\special{pa 3146 2464}%
\special{pa 3175 2479}%
\special{pa 3203 2494}%
\special{pa 3257 2528}%
\special{pa 3309 2566}%
\special{pa 3334 2587}%
\special{pa 3358 2609}%
\special{pa 3381 2632}%
\special{pa 3403 2656}%
\special{pa 3445 2706}%
\special{pa 3464 2733}%
\special{pa 3482 2760}%
\special{pa 3499 2787}%
\special{pa 3515 2816}%
\special{pa 3530 2844}%
\special{pa 3543 2874}%
\special{pa 3555 2903}%
\special{pa 3566 2933}%
\special{pa 3576 2964}%
\special{pa 3585 2995}%
\special{pa 3593 3026}%
\special{pa 3600 3057}%
\special{pa 3606 3088}%
\special{pa 3612 3120}%
\special{pa 3618 3151}%
\special{pa 3628 3215}%
\special{pa 3630 3230}%
\special{fp}%
%
\special{pn 8}%
\special{pa 3440 3040}%
\special{pa 3488 2998}%
\special{pa 3513 2977}%
\special{pa 3537 2957}%
\special{pa 3587 2917}%
\special{pa 3613 2897}%
\special{pa 3638 2878}%
\special{pa 3665 2860}%
\special{pa 3691 2842}%
\special{pa 3718 2825}%
\special{pa 3774 2793}%
\special{pa 3802 2778}%
\special{pa 3831 2763}%
\special{pa 3860 2750}%
\special{pa 3890 2737}%
\special{pa 3920 2725}%
\special{pa 3980 2703}%
\special{pa 4042 2685}%
\special{pa 4073 2677}%
\special{pa 4105 2670}%
\special{pa 4136 2663}%
\special{pa 4200 2653}%
\special{pa 4296 2644}%
\special{pa 4329 2643}%
\special{pa 4393 2643}%
\special{pa 4425 2644}%
\special{pa 4457 2646}%
\special{pa 4521 2652}%
\special{pa 4553 2657}%
\special{pa 4584 2662}%
\special{pa 4616 2667}%
\special{pa 4647 2673}%
\special{pa 4679 2680}%
\special{pa 4710 2686}%
\special{pa 4730 2690}%
\special{fp}%
%
\special{pn 8}%
\special{pa 4520 2880}%
\special{pa 4435 2745}%
\special{pa 4419 2717}%
\special{pa 4403 2690}%
\special{pa 4387 2662}%
\special{pa 4357 2604}%
\special{pa 4329 2546}%
\special{pa 4316 2516}%
\special{pa 4294 2456}%
\special{pa 4284 2425}%
\special{pa 4276 2394}%
\special{pa 4269 2363}%
\special{pa 4264 2331}%
\special{pa 4260 2299}%
\special{pa 4258 2267}%
\special{pa 4258 2235}%
\special{pa 4259 2203}%
\special{pa 4262 2171}%
\special{pa 4266 2138}%
\special{pa 4272 2106}%
\special{pa 4280 2075}%
\special{pa 4288 2043}%
\special{pa 4298 2012}%
\special{pa 4310 1981}%
\special{pa 4336 1921}%
\special{pa 4351 1892}%
\special{pa 4367 1864}%
\special{pa 4403 1810}%
\special{pa 4422 1785}%
\special{pa 4441 1759}%
\special{pa 4504 1687}%
\special{pa 4526 1664}%
\special{pa 4549 1641}%
\special{pa 4571 1618}%
\special{pa 4594 1595}%
\special{pa 4617 1573}%
\special{pa 4620 1570}%
\special{fp}%
%
\special{pn 8}%
\special{pa 3660 1250}%
\special{pa 3673 1279}%
\special{pa 3701 1337}%
\special{pa 3716 1365}%
\special{pa 3732 1393}%
\special{pa 3766 1447}%
\special{pa 3785 1474}%
\special{pa 3805 1499}%
\special{pa 3826 1524}%
\special{pa 3848 1548}%
\special{pa 3870 1571}%
\special{pa 3894 1594}%
\special{pa 3918 1616}%
\special{pa 3943 1636}%
\special{pa 3969 1656}%
\special{pa 3996 1675}%
\special{pa 4023 1692}%
\special{pa 4051 1709}%
\special{pa 4079 1725}%
\special{pa 4109 1739}%
\special{pa 4138 1752}%
\special{pa 4168 1764}%
\special{pa 4199 1775}%
\special{pa 4230 1785}%
\special{pa 4292 1801}%
\special{pa 4324 1807}%
\special{pa 4388 1817}%
\special{pa 4421 1820}%
\special{pa 4453 1822}%
\special{pa 4485 1823}%
\special{pa 4518 1823}%
\special{pa 4550 1822}%
\special{pa 4582 1820}%
\special{pa 4614 1817}%
\special{pa 4646 1813}%
\special{pa 4678 1808}%
\special{pa 4709 1803}%
\special{pa 4741 1798}%
\special{pa 4772 1792}%
\special{pa 4803 1785}%
\special{pa 4830 1780}%
\special{fp}%
%
\special{pn 4}%
\special{sh 1}%
\special{ar 1260 1750 16 16 0 6.2831853}%
\special{sh 1}%
\special{ar 1260 1750 16 16 0 6.2831853}%
%
\special{pn 4}%
\special{sh 1}%
\special{ar 640 2170 16 16 0 6.2831853}%
\special{sh 1}%
\special{ar 640 2170 16 16 0 6.2831853}%
%
\special{pn 4}%
\special{sh 1}%
\special{ar 1870 2170 16 16 0 6.2831853}%
\special{sh 1}%
\special{ar 1870 2170 16 16 0 6.2831853}%
%
\special{pn 4}%
\special{sh 1}%
\special{ar 640 2670 16 16 0 6.2831853}%
\special{sh 1}%
\special{ar 640 2670 16 16 0 6.2831853}%
%
\special{pn 4}%
\special{sh 1}%
\special{ar 1880 2670 16 16 0 6.2831853}%
\special{sh 1}%
\special{ar 1880 2670 16 16 0 6.2831853}%
%
\special{pn 4}%
\special{sh 1}%
\special{ar 1250 3080 16 16 0 6.2831853}%
\special{sh 1}%
\special{ar 1250 3080 16 16 0 6.2831853}%
%
\special{pn 4}%
\special{sh 1}%
\special{ar 3440 1810 16 16 0 6.2831853}%
\special{sh 1}%
\special{ar 3440 1810 16 16 0 6.2831853}%
%
\special{pn 4}%
\special{sh 1}%
\special{ar 3800 1490 16 16 0 6.2831853}%
\special{sh 1}%
\special{ar 3800 1490 16 16 0 6.2831853}%
%
\special{pn 8}%
\special{pa 2290 3270}%
\special{pa 2900 3270}%
\special{fp}%
\special{sh 1}%
\special{pa 2900 3270}%
\special{pa 2833 3250}%
\special{pa 2847 3270}%
\special{pa 2833 3290}%
\special{pa 2900 3270}%
\special{fp}%
%
\special{pn 13}%
\special{pa 650 2160}%
\special{pa 663 2131}%
\special{pa 677 2101}%
\special{pa 691 2073}%
\special{pa 707 2045}%
\special{pa 724 2019}%
\special{pa 744 1993}%
\special{pa 765 1969}%
\special{pa 788 1946}%
\special{pa 812 1924}%
\special{pa 837 1904}%
\special{pa 863 1885}%
\special{pa 891 1867}%
\special{pa 919 1851}%
\special{pa 947 1836}%
\special{pa 977 1822}%
\special{pa 1006 1810}%
\special{pa 1036 1799}%
\special{pa 1067 1789}%
\special{pa 1098 1780}%
\special{pa 1129 1772}%
\special{pa 1160 1766}%
\special{pa 1192 1760}%
\special{pa 1223 1755}%
\special{pa 1255 1751}%
\special{pa 1260 1750}%
\special{fp}%
%
\special{pn 13}%
\special{pa 3450 1820}%
\special{pa 3476 1802}%
\special{pa 3503 1783}%
\special{pa 3529 1765}%
\special{pa 3555 1746}%
\special{pa 3580 1727}%
\special{pa 3606 1707}%
\special{pa 3654 1665}%
\special{pa 3677 1644}%
\special{pa 3700 1621}%
\special{pa 3722 1598}%
\special{pa 3744 1574}%
\special{pa 3800 1510}%
\special{fp}%
%
\special{pn 8}%
\special{pa 1320 1710}%
\special{pa 1352 1707}%
\special{pa 1416 1699}%
\special{pa 1448 1696}%
\special{pa 1479 1691}%
\special{pa 1511 1687}%
\special{pa 1543 1682}%
\special{pa 1574 1677}%
\special{pa 1606 1671}%
\special{pa 1637 1665}%
\special{pa 1699 1651}%
\special{pa 1730 1642}%
\special{pa 1760 1634}%
\special{pa 1791 1624}%
\special{pa 1821 1615}%
\special{pa 1852 1604}%
\special{pa 1882 1594}%
\special{pa 1942 1572}%
\special{pa 1972 1560}%
\special{pa 2002 1549}%
\special{pa 2032 1537}%
\special{pa 2062 1526}%
\special{pa 2152 1490}%
\special{pa 2302 1435}%
\special{pa 2332 1425}%
\special{pa 2363 1415}%
\special{pa 2393 1406}%
\special{pa 2424 1396}%
\special{pa 2454 1387}%
\special{pa 2578 1355}%
\special{pa 2609 1348}%
\special{pa 2641 1341}%
\special{pa 2672 1334}%
\special{pa 2736 1322}%
\special{pa 2832 1307}%
\special{pa 2864 1303}%
\special{pa 2897 1299}%
\special{pa 2929 1296}%
\special{pa 2962 1294}%
\special{pa 2995 1291}%
\special{pa 3027 1290}%
\special{pa 3060 1289}%
\special{pa 3125 1289}%
\special{pa 3221 1295}%
\special{pa 3285 1303}%
\special{pa 3347 1315}%
\special{pa 3378 1322}%
\special{pa 3409 1331}%
\special{pa 3439 1340}%
\special{pa 3470 1350}%
\special{pa 3499 1361}%
\special{pa 3529 1373}%
\special{pa 3587 1399}%
\special{pa 3645 1427}%
\special{pa 3674 1442}%
\special{pa 3703 1456}%
\special{pa 3710 1460}%
\special{fp}%
%
\special{pn 8}%
\special{pa 700 2230}%
\special{pa 732 2227}%
\special{pa 764 2223}%
\special{pa 924 2208}%
\special{pa 955 2206}%
\special{pa 1019 2202}%
\special{pa 1083 2200}%
\special{pa 1178 2200}%
\special{pa 1242 2202}%
\special{pa 1274 2204}%
\special{pa 1305 2206}%
\special{pa 1337 2208}%
\special{pa 1369 2211}%
\special{pa 1401 2213}%
\special{pa 1432 2216}%
\special{pa 1464 2219}%
\special{pa 1496 2223}%
\special{pa 1560 2229}%
\special{pa 1624 2237}%
\special{pa 1656 2240}%
\special{pa 1688 2244}%
\special{pa 1752 2250}%
\special{pa 1784 2254}%
\special{pa 1848 2260}%
\special{pa 1881 2263}%
\special{pa 1913 2265}%
\special{pa 1946 2268}%
\special{pa 1978 2270}%
\special{pa 2011 2272}%
\special{pa 2043 2273}%
\special{pa 2076 2275}%
\special{pa 2108 2276}%
\special{pa 2141 2276}%
\special{pa 2173 2277}%
\special{pa 2206 2277}%
\special{pa 2238 2276}%
\special{pa 2270 2276}%
\special{pa 2303 2274}%
\special{pa 2335 2273}%
\special{pa 2367 2271}%
\special{pa 2399 2268}%
\special{pa 2431 2266}%
\special{pa 2527 2254}%
\special{pa 2589 2244}%
\special{pa 2621 2238}%
\special{pa 2683 2224}%
\special{pa 2713 2216}%
\special{pa 2744 2208}%
\special{pa 2774 2199}%
\special{pa 2834 2179}%
\special{pa 2864 2168}%
\special{pa 2894 2156}%
\special{pa 2952 2132}%
\special{pa 2982 2119}%
\special{pa 3011 2106}%
\special{pa 3039 2092}%
\special{pa 3068 2078}%
\special{pa 3097 2063}%
\special{pa 3125 2049}%
\special{pa 3154 2033}%
\special{pa 3210 2003}%
\special{pa 3239 1987}%
\special{pa 3323 1939}%
\special{pa 3351 1922}%
\special{pa 3379 1906}%
\special{pa 3390 1900}%
\special{fp}%
%
\special{pn 8}%
\special{pa 3710 1460}%
\special{pa 3670 1400}%
\special{fp}%
%
\special{pn 8}%
\special{pa 3720 1460}%
\special{pa 3650 1490}%
\special{fp}%
%
\special{pn 8}%
\special{pa 3390 1900}%
\special{pa 3300 1920}%
\special{fp}%
%
\special{pn 8}%
\special{pa 3400 1900}%
\special{pa 3320 1990}%
\special{fp}%
\put(12.6000,-32.7000){\makebox(0,0){$D^2$}}%
\put(26.1000,-34.1000){\makebox(0,0){$\phi$}}%
\put(4.3000,-21.8000){\makebox(0,0){$p_{j-1}$}}%
\put(12.0000,-15.9000){\makebox(0,0){$p_j$}}%
\put(32.3000,-17.3000){\makebox(0,0){$a_{j-1}$}}%
\put(40.0000,-14.9000){\makebox(0,0){$a_j$}}%
\put(40.4000,-32.7000){\makebox(0,0){$M$}}%
\put(40.0000,-10.7000){\makebox(0,0){$\mathcal{L}_j$}}%
\put(50.8000,-17.1000){\makebox(0,0){$\mathcal{L}_{j+1}$}}%
\put(28.3000,-27.4000){\makebox(0,0){$\mathcal{L}_{j-1}$}}%
%
\special{pn 4}%
\special{sh 1}%
\special{ar 4400 1820 16 16 0 6.2831853}%
\special{sh 1}%
\special{ar 4400 1820 16 16 0 6.2831853}%
%
\special{pn 4}%
\special{sh 1}%
\special{ar 4380 2640 16 16 0 6.2831853}%
\special{sh 1}%
\special{ar 4380 2640 16 16 0 6.2831853}%
%
\special{pn 4}%
\special{sh 1}%
\special{ar 3560 2950 16 16 0 6.2831853}%
\special{sh 1}%
\special{ar 3560 2950 16 16 0 6.2831853}%
%
\special{pn 4}%
\special{sh 1}%
\special{ar 3190 2490 16 16 0 6.2831853}%
\special{sh 1}%
\special{ar 3190 2490 16 16 0 6.2831853}%
\end{picture}}%
\caption{A rough picture of a pseudo holomorphic map $\phi : D^2\rightarrow M$}
\end{center}
\end{figure}
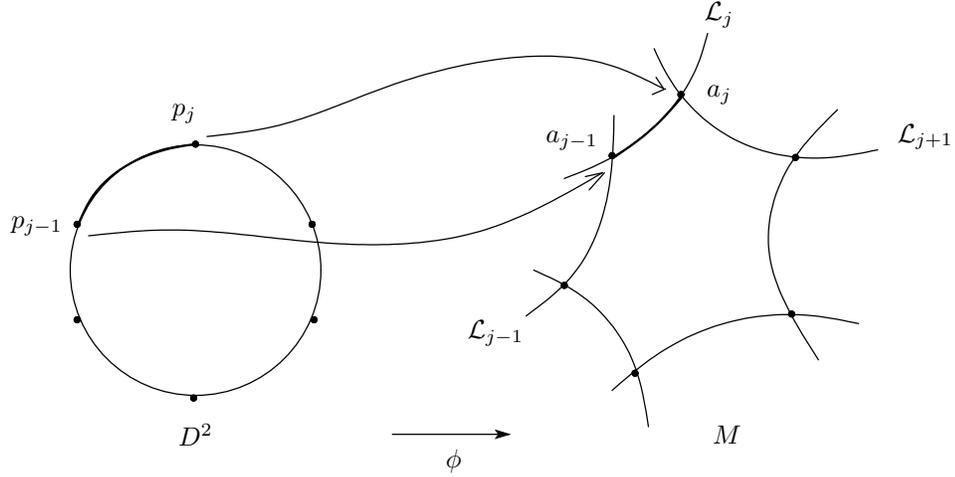

Now we turn to the case $M=(T^2,\omega )$ and consider the Fukaya category Fuk($T^2,\omega $). The Lagrangian submanifolds defined in section 2 are actually special Lagrangian submanifolds of $(T^2,\omega )$.

As we saw Lemma \ref{lem4.1}, the isomorphism $\tilde{\phi } : E_{(\frac{1}{2},\frac{\eta }{2})}\rightarrow C(\psi )$ exists if and only if $F(\psi )=0$, where $F(\psi )$ is the composition of the morphism $E_{(\frac{1}{2},\frac{\eta }{2})}\rightarrow E_{(1,\nu )}$ and $\psi : E_{(1,\nu )}\rightarrow TE_{(0,\mu )}$. So, let us first consider the corresponding product $m_2 : \mathcal{C}^0(s_{(\frac{1}{2},\frac{\eta }{2})},s_{(1,\nu )})\otimes \mathcal{C}^1(s_{(1,\nu )},s_{(0,\mu )})\rightarrow \mathcal{C}^1(s_{(\frac{1}{2},\frac{\eta }{2})},s_{(0,\mu )})$, where 
\begin{align*}
&s_{(\frac{1}{2},\frac{\eta }{2})}=(\mathcal{L}_{(\frac{1}{2},\frac{u}{2})},\mathscr{L}_{(\frac{1}{2},\frac{\eta }{2})},\nabla_{\mathscr{L}_{(\frac{1}{2},\frac{\eta }{2})}}), \ s_{(1,\nu )}=(\mathcal{L}_{(1,s)},\mathscr{L}_{(1,\nu )},\nabla_{\mathscr{L}_{(1,\nu )}}), \\
&s_{(0,\mu )}=(\mathcal{L}_{(0,p)},\mathscr{L}_{(0,\mu )},\nabla_{\mathscr{L}_{(0,\mu )}}),
\end{align*}
and give a geometric interpretation of it. Here, note that the elements of $\mathcal{L}_{(\frac{1}{2},\frac{u}{2})}\cap \mathcal{L}_{(1,s)}$, $\mathcal{L}_{(1,s)}\cap \mathcal{L}_{(0,p)}$ and $\mathcal{L}_{(\frac{1}{2},\frac{u}{2})}\cap \mathcal{L}_{(0,p)}$ are only one point, respectively. We consider the following setting in the fundamental domain [0,2$\pi $]$\times$[0,2$\pi $] $\subset\mathbb{R}^2$ of $T^2$. Let $e_1$ be the intersection point of $\mathcal{L}_{(\frac{1}{2},\frac{u}{2})}$ and $\mathcal{L}_{(1,s)}$. Similarly, let $e_2$, $e_3$ be the intersection points of $\mathcal{L}_{(1,s)}$ and $\mathcal{L}_{(0,p)}$, $\mathcal{L}_{(\frac{1}{2},\frac{u}{2})}$ and $\mathcal{L}_{(0,p)}$, respectively. We regard $e_1$, $e_2$ and $e_3$ as elements of $\mathcal{C}^0(s_{(\frac{1}{2},\frac{\eta }{2})},s_{(1,\nu )})$, $\mathcal{C}^1(s_{(1,\nu )},s_{(0,\mu )})$ and $\mathcal{C}^1(s_{(\frac{1}{2},\frac{\eta }{2})},s_{(0,\mu )})$, respectively. Then, the product $m_2 : \mathcal{C}^0(s_{(\frac{1}{2},\frac{\eta }{2})},s_{(1,\nu )})\otimes \mathcal{C}^1(s_{(1,\nu )},s_{(0,\mu )})\rightarrow \mathcal{C}^1(s_{(\frac{1}{2},\frac{\eta }{2})},s_{(0,\mu )})$ turns out to be
\begin{align}
&m_2(e_1\otimes e_2) \notag \\
&=c(e_1,e_2,e_3)e_3 \notag \\ 
&=\pm \displaystyle{\sum_{n\in \mathbb{Z}}}e^{\frac{\mathbf{i}}{2\pi }\cdot (-\frac{1}{\tau })\cdot \frac{\pi ^2}{2}(2n+1)^2+\frac{\mathbf{i}}{2}(q+t+\pi )(2n+1)-\frac{\mathbf{i}}{2}t(2n+1)-\frac{\mathbf{i}}{2}q(2n+1)}e_3. \label{m2}
\end{align}
In the formula (\ref{m2}), the values $-\frac{1}{\tau }\cdot \frac{\pi ^2}{2}(2n+1)^2$ ($n\in \mathbb{Z}$) are the symplectic areas of the triangles surrounded by the Lagrangian submanifolds $\mathcal{L}_{(0,p)}$, $\mathcal{L}_{(1,s)}$, $\mathcal{L}_{(\frac{1}{2},\frac{u}{2})}$ and their copies in the covering space $\mathbb{R}^2$ of $T^2$ with the complexified symplectic form $\omega =-\frac{1}{\tau }dx\wedge dy$. For example, in Figure 3, where $p=\pi $, $s=0$, $u=0$, $v\equiv q+t+\pi \ (\mathrm{mod}\ 2\pi \mathbb{Z})$, both areas of triangles represented by shaded areas are $\frac{\pi ^2}{2}$, and these values correspond to $\frac{\pi ^2}{2}(2n+1)^2$ with $n=0,-1$. Similarly, both areas of triangles surrounded by the bold lines are $\frac{9}{2}\pi ^2$, and these values correspond to $\frac{\pi ^2}{2}(2n+1)^2$ with $n=1,-2$. The remaining part 
\begin{equation*}
e^{\frac{\mathbf{i}}{2}(q+t+\pi )(2n+1)-\frac{\mathbf{i}}{2}t(2n+1)-\frac{\mathbf{i}}{2}q(2n+1)}\ (n\in \mathbb{Z})
\end{equation*}
in the formula (\ref{m2}) are the results of the calculations of the path-ordered integrations along the boundary of those triangles with the connections $\nabla_{\mathscr{L}_{(0,\mu )}}$, $\nabla_{\mathscr{L}_{(1,\nu )}}$, $\nabla_{\mathscr{L}_{(\frac{1}{2},\frac{\eta }{2})}}$ ($\eta \equiv \mu +\nu +\pi +\pi \tau \ (\mathrm{mod}\ 2\pi (\mathbb{Z}\oplus \tau \mathbb{Z}))$). Here, by comparing the formula (\ref{m2}) to the theta function 
\begin{equation*}
\vartheta _{\frac{1}{2},\frac{1}{2}}(z,\tau )=\sum_{n\in \mathbb{Z}} e^{\pi \mathbf{i}(n+\frac{1}{2})^2\tau +2\pi \mathbf{i}(n+\frac{1}{2})(z+\frac{1}{2})},
\end{equation*}
we see that the structure constant in the formula (\ref{m2}) turns out to be 
\begin{equation*}
\pm \sum_{n\in \mathbb{Z}}e^{\pi \mathbf{i}(n+\frac{1}{2})^2\cdot (-\frac{1}{\tau })+\pi \mathbf{i}(n+\frac{1}{2})}=\pm \vartheta _{\frac{1}{2},\frac{1}{2}} \left( 0,-\frac{1}{\tau } \right) =0.
\end{equation*}
This fact can also be understood in the Fukaya category Fuk($T^2,\omega $) as the cancellation of the signed sum of the exponentials of the symplectic areas of those triangles. Actually, for two symmetric triangles described in Figure 3, different signs are assigned, respectively. Thus, $m_2(e_1\otimes e_2)=0$. 
\begin{figure}
\begin{center}
\begin{center}
{\unitlength 0.1in%
\begin{picture}(28.2000,17.7500)(3.9000,-20.0000)%
%
\special{pn 8}%
\special{pa 400 1200}%
\special{pa 3190 1200}%
\special{fp}%
\special{sh 1}%
\special{pa 3190 1200}%
\special{pa 3123 1180}%
\special{pa 3137 1200}%
\special{pa 3123 1220}%
\special{pa 3190 1200}%
\special{fp}%
%
\special{pn 8}%
\special{pa 1800 2000}%
\special{pa 1800 400}%
\special{fp}%
\special{sh 1}%
\special{pa 1800 400}%
\special{pa 1780 467}%
\special{pa 1800 453}%
\special{pa 1820 467}%
\special{pa 1800 400}%
\special{fp}%
%
\special{pn 8}%
\special{pa 400 1900}%
\special{pa 3200 500}%
\special{fp}%
%
\special{pn 8}%
\special{pa 1000 2000}%
\special{pa 2600 400}%
\special{fp}%
%
\special{pn 8}%
\special{pa 390 1000}%
\special{pa 3200 1000}%
\special{fp}%
%
\special{pn 8}%
\special{pa 400 1400}%
\special{pa 3210 1400}%
\special{fp}%
%
\special{pn 8}%
\special{pa 390 600}%
\special{pa 3200 600}%
\special{fp}%
%
\special{pn 8}%
\special{pa 400 1800}%
\special{pa 3210 1800}%
\special{fp}%
%
\special{pn 4}%
\special{pa 2040 1080}%
\special{pa 1980 1020}%
\special{fp}%
\special{pa 2080 1060}%
\special{pa 2020 1000}%
\special{fp}%
\special{pa 2120 1040}%
\special{pa 2080 1000}%
\special{fp}%
\special{pa 2000 1100}%
\special{pa 1950 1050}%
\special{fp}%
\special{pa 1960 1120}%
\special{pa 1920 1080}%
\special{fp}%
\special{pa 1920 1140}%
\special{pa 1890 1110}%
\special{fp}%
%
\special{pn 4}%
\special{pa 1620 1380}%
\special{pa 1560 1320}%
\special{fp}%
\special{pa 1650 1350}%
\special{pa 1600 1300}%
\special{fp}%
\special{pa 1680 1320}%
\special{pa 1640 1280}%
\special{fp}%
\special{pa 1710 1290}%
\special{pa 1680 1260}%
\special{fp}%
\special{pa 1740 1260}%
\special{pa 1720 1240}%
\special{fp}%
\special{pa 1580 1400}%
\special{pa 1520 1340}%
\special{fp}%
\special{pa 1520 1400}%
\special{pa 1480 1360}%
\special{fp}%
\put(32.7000,-12.0000){\makebox(0,0){$x$}}%
\put(18.0000,-3.0000){\makebox(0,0){$y$}}%
%
\special{pn 20}%
\special{pa 600 1800}%
\special{pa 2990 600}%
\special{fp}%
%
\special{pn 20}%
\special{pa 3000 600}%
\special{pa 2400 600}%
\special{fp}%
%
\special{pn 20}%
\special{pa 2400 600}%
\special{pa 1200 1800}%
\special{fp}%
%
\special{pn 20}%
\special{pa 1200 1800}%
\special{pa 600 1800}%
\special{fp}%
\put(18.9000,-12.9000){\makebox(0,0){$O$}}%
\put(17.2000,-9.3000){\makebox(0,0){$\pi$}}%
\put(16.8000,-5.1000){\makebox(0,0){$3\pi$}}%
\put(16.5000,-15.1000){\makebox(0,0){$-\pi$}}%
\put(16.2000,-18.9000){\makebox(0,0){$-3\pi$}}%
\put(34.4000,-10.0000){\makebox(0,0){$\mathcal{L}_{(0,\pi)}$}}%
\put(26.5000,-2.9000){\makebox(0,0){$\mathcal{L}_{(1,0)}$}}%
\put(33.9000,-4.2000){\makebox(0,0){$\mathcal{L}_{(\frac{1}{2},0)}$}}%
\put(23.4000,-15.6000){\makebox(0,0){$\frac{\pi^2}{2}$}}%
\put(10.0000,-8.0000){\makebox(0,0){$\frac{9}{2}\pi^2$}}%
%
\special{pn 8}%
\special{pa 1010 890}%
\special{pa 1010 1690}%
\special{dt 0.045}%
\special{sh 1}%
\special{pa 1010 1690}%
\special{pa 1030 1623}%
\special{pa 1010 1637}%
\special{pa 990 1623}%
\special{pa 1010 1690}%
\special{fp}%
%
\special{pn 8}%
\special{pa 2200 1510}%
\special{pa 1700 1350}%
\special{dt 0.045}%
\special{sh 1}%
\special{pa 1700 1350}%
\special{pa 1757 1389}%
\special{pa 1751 1366}%
\special{pa 1770 1351}%
\special{pa 1700 1350}%
\special{fp}%
%
\special{pn 8}%
\special{pa 2270 1510}%
\special{pa 2080 1110}%
\special{dt 0.045}%
\special{sh 1}%
\special{pa 2080 1110}%
\special{pa 2091 1179}%
\special{pa 2103 1158}%
\special{pa 2127 1162}%
\special{pa 2080 1110}%
\special{fp}%
%
\special{pn 8}%
\special{pa 1170 800}%
\special{pa 2360 800}%
\special{dt 0.045}%
\special{sh 1}%
\special{pa 2360 800}%
\special{pa 2293 780}%
\special{pa 2307 800}%
\special{pa 2293 820}%
\special{pa 2360 800}%
\special{fp}%
\end{picture}}%
\end{center}
\caption{$p=\pi $, $s=0$, $u=0$, $v\equiv q+t+\pi \ (\mathrm{mod}\ 2\pi \mathbb{Z})$}
\end{center}
\end{figure}
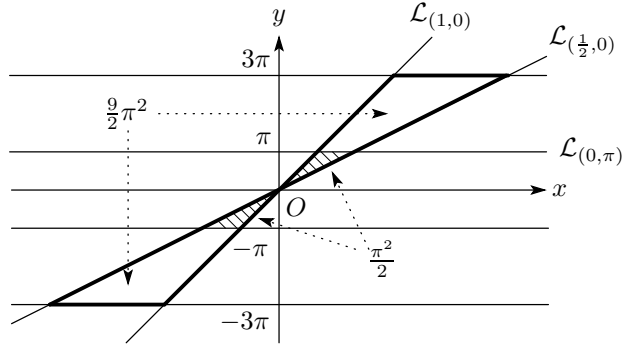

Next, we consider an interpretation for the value $c_{\tau }$. Recall that the holomorphic maps $\tilde{\phi } : E_{(\frac{1}{2},\frac{\eta }{2})}\rightarrow C(\psi )$ and $\phi : C(\psi )\rightarrow E_{(\frac{1}{2},\frac{\eta }{2})}$ satisfy the relation $\phi \tilde{\phi } =c_{\tau }I_2$, where
\begin{align*}
c_{\tau }&=\frac{\bar{\tau } -\tau }{4\pi \tau }\mathbf{i}\displaystyle{\sum_{l\in \mathbb{Z}}}(-1)^l \cdot \pi \left( 2l+1 \right) e^{\frac{\mathbf{i}}{2\pi }\cdot (-\frac{1}{\tau })\cdot \frac{\pi ^2}{2}(2l+1)^2}\\
&=\left.-\frac{\bar{\tau } -\tau }{4\pi \tau }\mathbf{i}\frac{\partial }{\partial \beta }\vartheta _{\frac{1}{2},\frac{1}{2}}\left(\beta ,-\frac{1}{\tau }\right)\right|_{\beta =0}.
\end{align*}
Similarly as in the case of the product $m_2(e_1\otimes e_2)$, the values $-\frac{1}{\tau }\cdot \frac{\pi ^2}{2}(2l+1)^2\ (l\in \mathbb{Z})$ are the symplectic areas of the triangles surrounded by the Lagrangian submanifolds $\mathcal{L}_{(0,p)}$, $\mathcal{L}_{(1,s)}$, $\mathcal{L}_{(\frac{1}{2},\frac{u}{2})}$ and their copies in the covering space $\mathbb{R}^2$ of $T^2$ with the complexified symplectic form $\omega =-\frac{1}{\tau }dx\wedge dy$. We explain the values $\pi (2l+1)\ (l\in \mathbb{Z})$ by using Figure 3. In Figure 3, both length of edges which are parallel to $x$-axis in the triangles represented by shaded areas are $\pi $, and these values correspond to $|\pi (2l+1)|$ with $l=0,-1$. Similarly, both length of edges which are parallel to $x$-axis in the triangles surrounded by the bold lines are $3\pi $, and these values correspond to $|\pi (2l+1)|$ with $l=1,-2$. Thus, the values $\pi (2l+1)$ ($l\in \mathbb{Z}$) give the information about the length of edges which correspond to self intersecting Lagrangian submanifolds. In fact, the value 
\begin{equation*}
\displaystyle{\sum_{l\in \mathbb{Z}}}(-1)^l \cdot \pi \left( 2l+1 \right) e^{\frac{\mathbf{i}}{2\pi }\cdot (-\frac{1}{\tau })\cdot \frac{\pi ^2}{2}(2l+1)^2}=\left.-\frac{\partial }{\partial \beta }\vartheta _{\frac{1}{2},\frac{1}{2}}\left(\beta ,-\frac{1}{\tau }\right)\right|_{\beta =0}
\end{equation*}
in $c_{\tau }$ is the structure constant of the non-transversal $A_{\infty }$ product $m_3$ of $e_1$, $e_2$, $e_3$ in \cite{Kaj} (eq.(19) in the case of $n=2$ and $b=1$). 

\section{The $SL(2;\mathbb{Z})$ action}
In section 4, we mentioned that the exact triangle 
\begin{equation*}
\begin{CD}
\cdots E_{(\frac{a}{n},\frac{\mu }{n})} @>>> C(\psi ') @>>> E_{(\frac{b}{m},\frac{\nu }{m})} @>\psi '\not=0 >> TE_{(\frac{a}{n},\frac{\mu }{n})} \cdots ,
\end{CD}
\end{equation*}
where $\frac{a}{n}<\frac{b}{m}$, i.e., $bn-am>0$ and dim$\mathrm{Ext}^1(E_{(\frac{b}{m},\frac{\nu }{m})},E_{(\frac{a}{n},\frac{\mu }{n})})=bn-am=1$ becomes the exact triangle
\begin{equation*}
\begin{CD}
\cdots E_{(0,\mu )} @>>> C(\psi ) @>>> E_{(1,\nu )} @> \psi \not=0 >> TE_{(0,\mu )} \cdots
\end{CD}
\end{equation*}
by considering the $SL(2;\mathbb{Z})$ action on $(T^2,\omega )$. In this section, we explain this fact by using the homological mirror symmetry of two tori, and check that $C(\psi ')\cong E_{(\frac{a+b}{m+n},\frac{\eta }{m+n})}$ if and only if $\eta \equiv \mu +\nu +\pi +\pi \tau \ (\mathrm{mod}\ 2\pi (\mathbb{Z}\oplus \tau \mathbb{Z}))$.

We explain the $SL(2;\mathbb{Z})$ action on $(T^2,\omega )$. For an element 
\begin{equation*}
\left( \begin{array}{ccc} g_{11}&g_{12}\\g_{21}&g_{22} \end{array} \right) \ (g_{11}, \ g_{12}, \ g_{21}, \ g_{22}\in \mathbb{Z}, \ g_{11}g_{22}-g_{12}g_{21}=1)
\end{equation*}
in $SL(2;\mathbb{Z})$, the $SL(2;\mathbb{Z})$ action on $\mathbb{R}^2$ is defined by
\begin{equation*}
\left( \begin{array}{ccc} x\\y \end{array} \right) \in \mathbb{R}^2 \longmapsto \left( \begin{array}{ccc} g_{11}&g_{12}\\g_{21}&g_{22} \end{array} \right) \left( \begin{array}{ccc} x\\y \end{array} \right) \in \mathbb{R}^2.
\end{equation*}
This action induces the $SL(2;\mathbb{Z})$ action on $T^2\cong \mathbb{R}^2/\mathbb{Z}^2$. By using the above matrix, we define an automorphism $\varphi : (T^2,\omega )\rightarrow (T^2,\omega )$ by
\begin{equation*}
\varphi \left( \begin{array}{ccc} x\\y \end{array} \right) = \left( \begin{array}{ccc} g_{11}&g_{12}\\g_{21}&g_{22} \end{array} \right) \left( \begin{array}{ccc} x\\y \end{array} \right).
\end{equation*}
Then, we can check easily that the automorphism $\varphi $ preserves the symplectic structure $\omega $, i.e., $\varphi ^*\omega =\omega $. Therefore, the automorphism $\varphi $ is a symplectic automorphism. Moreover, by using this symplectic automorphism $\varphi $, we can transform a pair $(L,\mathscr{E})$ of a Lagrangian submanifold $L$ in $(T^2,\omega )$ and a (flat) vector bundle $\mathscr{E}\rightarrow L$ as follows. Now $\varphi $ is invertible and $\varphi ^{-1}(L)$ is also a Lagrangian submanifold in $(T^2,\omega )$. Then, $\varphi $ induces a vector bundle $\varphi ^*\mathscr{E}\rightarrow \varphi ^{-1}(L)$. 
\begin{equation*}
\begin{CD}
\varphi ^* \mathscr{E} @>>> \mathscr{E}\\
@VVV                                 @VVV\\
\varphi ^{-1}(L) @>>\varphi > L\ .
\end{CD}
\end{equation*}
Therefore, by considering the symplectic automorphism $\varphi $, the pair $(L,\mathscr{E})$ is mapped to the pair $(\varphi ^{-1}(L),\varphi ^*\mathscr{E})$. 

Let us consider the matrix 
\begin{equation*}
\left(\begin{array}{ccc}n&m-n\\a&b-a\end{array}\right).
\end{equation*}
Since we assume $bn-am=1$, we see that this matrix is an element in $SL(2;\mathbb{Z})$. By using this matrix, we define a symplectic automorphism $\varphi : (T^2,\omega )\rightarrow (T^2,\omega )$ by 
\begin{equation*}
\varphi \left( \begin{array}{ccc}x\\y\end{array} \right) = \left( \begin{array}{ccc}n&m-n\\a&b-a\end{array} \right) \left( \begin{array}{ccc}x\\y\end{array} \right).
\end{equation*}
By using this symplectic automorphism $\varphi : (T^2,\omega )\rightarrow (T^2,\omega )$, the objects 
\begin{equation*}
s_{(\frac{a}{n},\frac{\mu }{n})}=(\mathcal{L}_{(\frac{a}{n},\frac{p}{n})},\mathscr{L}_{(\frac{a}{n},\frac{\mu }{n})},\nabla_{\mathscr{L}_{(\frac{a}{n},\frac{\mu }{n})}}),\ s_{(\frac{b}{m},\frac{\nu }{m})}=(\mathcal{L}_{(\frac{b}{m},\frac{s}{m})},\mathscr{L}_{(\frac{b}{m},\frac{\nu }{m})},\nabla_{\mathscr{L}_{(\frac{b}{m},\frac{\nu }{m})}})
\end{equation*}
of the Fukaya category Fuk($T^2,\omega $) are mapped to the objects 
\begin{equation*}
s_{(0,\mu )}=(\mathcal{L}_{(0,p)},\mathscr{L}_{(0,\mu )},\nabla_{\mathscr{L}_{(0,\mu )}}),\ s_{(1,\nu )}=(\mathcal{L}_{(1,s)},\mathscr{L}_{(1,\nu )},\nabla_{\mathscr{L}_{(1,\nu )}}),
\end{equation*}
respectively. Hence, in the triangulated category $Tr(\mathrm{Fuk}(T^2,\omega ))$ obtained by the $A_{\infty }$ category Fuk$(T^2,\omega )$, the exact triangle
\begin{equation*}
\begin{CD}
\cdots s_{(\frac{a}{n},\frac{\mu }{n})} @>>> C(\Psi ') @>>> s_{(\frac{b}{m},\frac{\nu }{m})} @> \Psi '\not=0 >> Ts_{(\frac{a}{n},\frac{\mu }{n})} \cdots
\end{CD}
\end{equation*}
becomes the exact triangle
\begin{equation*}
\begin{CD}
\cdots s_{(0,\mu )} @>>> C(\Psi ) @>>> s_{(1,\nu )} @> \Psi \not=0 >> Ts_{(0,\mu )} \cdots
\end{CD}
\end{equation*}
by this $SL(2;\mathbb{Z})$ action. Note that the complex structure of the mirror dual complex torus $\check{T}^2$ is also preserved when we consider the $SL(2;\mathbb{Z})$ action on $(T^2,\omega )$. Thus, by using the homological mirror symmetry $Tr(\mathrm{Fuk}(T^2,\omega ))\cong Tr(DG_{\check{T}^2})$ (see \cite{P}, \cite{P A}, \cite{abouzaid} etc.), the exact triangle
\begin{equation*}
\begin{CD}
\cdots E_{(\frac{a}{n},\frac{\mu }{n})} @>>> C(\psi ') @>>> E_{(\frac{b}{m},\frac{\nu }{m})} @>\psi '\not=0 >> TE_{(\frac{a}{n},\frac{\mu }{n})} \cdots 
\end{CD}
\end{equation*}
becomes the exact triangle
\begin{equation*}
\begin{CD}
\cdots E_{(0,\mu )} @>>> C(\psi ) @>>> E_{(1,\nu )} @>\psi \not=0 >> TE_{(0,\mu )} \cdots 
\end{CD}
\end{equation*}
in $Tr(DG_{\check{T}^2})$. This fact implies that $C(\psi ')\cong E_{(\frac{a+b}{m+n},\frac{\eta }{m+n})}$ if and only if $\eta \equiv \mu +\nu +\pi +\pi \tau \ (\mathrm{mod}\ 2\pi (\mathbb{Z}\oplus \tau \mathbb{Z}))$. 

\section*{Acknowledgment}
I would like to thank Hiroshige Kajiura for various advices in writing this paper. I am also grateful to referees for useful suggestions.


\begin{thebibliography}{99}
\bibitem{abouzaid}
M. Abouzaid, I. Smith, Homological mirror symmetry for the four-torus, Duke Mathematical Journal, 152.3 (2010), 373-440.
\bibitem{atiyah}
M. F. Atiyah, Vector bundles over an elliptic curve. Proc. London Math. Soc, 1957.
\bibitem{bondal}
A. Bondal and M. Kapranov, Enhanced triangulated categories. Math. USSR Sbornik 70:93-107, 1991.
\bibitem{Fuk}
K. Fukaya, Morse homotopy, $A^{\infty }$-category, and Floer homologies. In: Proceedings of GARC Workshop on Geometry and Topology '93 (Seoul, 1993). Lecture Notes in Series, vol. 18, pp. 1-102. Seoul Nat. Univ., Seoul (1993).
\bibitem{Fuk t}
K. Fukaya, Mirror symmetry of abelian varieties and multi theta functions. J. Algebr. Geom. 11, 393-512 (2002).
\bibitem{Kaj}
H. Kajiura, Homological perturbation theory and homological mirror symmetry. In Higher Structures in Geometry and Physics, volume 287 of Progress in Math., pages 201-226, Birkh\"{a}user/Springer, New York, 2011.
\bibitem{line}
H. Kajiura, On some deformation of fukaya categories. Symplectic, Poisson, and Noncommutative Geometry, 93-130, MSRI Publ. 62, Cambridge Univ. Press, New York, 2014.
\bibitem{K}
M. Kontsevich, Homological algebra of mirror symmetry. In Proceedings of the International Congress of Mathematicians, Vol. 1, 2 (Z\"{u}rich, 1994), volume 184, pages 120-139. Birkh\"{a}user, 1995. math.AG/9411018.
\bibitem{dg}
M. Kontsevich, Y. Soibelman, Homological mirror symmetry and torus fibrations. In Symplectic geometry and mirror symmetry (Seoul, 2000), pages 203-263. World Sci.Publishing, River Edge, NJ, 2001. math.SG/0011041.
\bibitem{Le}
N. C. Leung, S-T. Yau, E. Zaslow, From special Lagrangian to Hermitian-Yang-Mills via Fourier-Mukai transform. Adv. Theor. Math. Phys. 4:13191341, 2000.
\bibitem{t}
D. Mumford, Tata Lectures on Theta, Progress in Mathmatics Vol 28, Birkh\"{a}user BostonEBaselEStuttgart, 1983.
\bibitem{P A}
A. Polishchuk, $A_{\infty }$-structures on an elliptic curve. Comm. Math. Phys. 247, 527 (2004). math.AG/0001048.
\bibitem{Abelian}
A. Polishchuk, Abelian Varieties, Theta Functions and the Fourier Transform, Cambridge University Press, 2003.
\bibitem{P}
A. Polishchuk, E. Zaslow, Categorical mirror symmetry : the elliptic curve. Adv. Theor. Math. Phys. 2, 443-470 (1998). math.AG/9801119.
\bibitem{pic-lef}
P. Seidel, Fukaya categories and Picard-Lefschetz theory, European Math. Soc., 2008.
\bibitem{syz}
A.Strominger, S.-T.Yau, and E. Zaslow, Mirror Symmetry is T-Duality. Nucl. Phys. B, 479:243-259, 1996.
\end{thebibliography}
\end{document}